\documentclass[12pt,letterpaper]{article}
\usepackage[margin=1in]{geometry}
\usepackage{graphicx}
\usepackage{bbding}
\usepackage{perpage}
\usepackage{mathtools}
\usepackage{multicol}
\usepackage{calc}
\usepackage{amsmath, ulem}
\usepackage{amsxtra}
\usepackage{enumerate}
\usepackage[shortlabels]{enumitem}
\usepackage[thicklines]{cancel}
\usepackage{amssymb,wasysym}
\usepackage{amsfonts}
\usepackage{pifont}

\usepackage{amsthm}
\usepackage[sans]{dsfont}
\usepackage{xcolor,cancel}
\usepackage{verbatim}

 \definecolor{lightgrey}{rgb}{0.7,0.7,0.7}
\usepackage{textcomp}
\usepackage{mathrsfs}
\usepackage{mathrsfs}
\usepackage[colorlinks]{hyperref}
\usepackage{amssymb} 
\usepackage{lmodern}

\graphicspath{{figday/}}

\newcommand{\bP}{\mathbf{P}}

\newtheorem{thm}{Theorem}[section] 
\newtheorem{lem}[thm]{Lemma}

\newtheorem{rmk}[thm]{Remark}

\theoremstyle{plain}

\title{Remarks on Lifespan and Continuation Criterion of Two Dimensional Incompressible Fluid Models}
\author{Anping Pan }

\begin{document}

\maketitle

\begin{abstract}
In this paper, we study the lifespan of several two-dimensional incompressible fluid models. Motivated by a novel energy-vorticity formulation, combining linear transport estimate and a bootstrap argument, we are able to show several longtime existence results of two dimensional fluid models in close-to-Euler regime. A byproduct of our approach is a new conditional BKM type result for the inhomogeneous Euler equation.
\end{abstract}

\tableofcontents

\section{Introduction}

The motion of incompressible fluids is a ubiquitous phenomenon in the real world and is a central theme in modern PDE research. The incompressible Euler equation, describing the motion of incompressible fluid flow in a vessel, admits both fruitful geometric structure and delicate mathematical challenges. The global well-posedness issue of Euler equation in $3$-d boundaryless domain ($\mathbb R^3$ or $\mathbb T^3$) with smooth initial data remains open and is closely related to the Clay Milliennium Prize problem concerning Navier-Stokes equation. However, many significant progresses have been made in past few years regarding finite-time singularity formation of Euler equation, either from initial data with limited regularity, or in domain with boundaries. For extensive discussions on this topic, we refer the interested reader to \cite{[CH21]}\cite{ChenHou2025PNAS}\cite{Elgindi2021}\cite{LuoHou2014} and references therein.

Despite the difficulty of understanding the global well-posedness issue in $3$-d, the classical Yudovich theory \cite{[Yud63]} proved that the incompressible Euler equation in $2$-d is globally well-posed as long as the initial vorticity $\omega_0=\nabla\times u_0=\partial_1u_0^2-\partial_2u_0^1$ is bounded. Regarding partial regularity results in $3$-d, in the seminal work \cite{BealeKatoMajda1984}, the authors established a continuation/blow up criteria (known as BKM criteria) for incompressible Euler in $3$-d, using the $L_t^1L_x^\infty$-norm of vorticity $\omega$. More precisely, solution of the incompressible Euler equation can be continued past time $T$ provided:
\begin{equation}\label{BKM L1Linfty Bound}
\int_0^T \lVert\omega(t)\rVert_\infty dt<\infty .   
\end{equation}

From a geometric and variational viewpoint, the Euler equation and many other incompressible fluid PDEs are Hamiltonian equations falling in the category of Euler-Poincar\'e equations on the dual Lie algebra of the infinite dimensional Lie group of volume-preserving diffeomorphisms, and can be derived from various action principles. A non-exhaustive list of such equations includes the ideal Boussinesq equation, ideal magneto-hydrodynamics (MHD in short), inhomogeneous Euler equation (IIE in short), Euler-$\alpha$ equation, generalized SQG equation, Hall magneto-hydrodynamics and many more. We refer the interested readers to \cite{Ar66}\cite{AK98}\cite{HMR98}\cite{HSS09}\cite{MR99} for further discussions in this direction. 

Similar to the Lagrangian-Hamiltonian correspondence in classical mechanics, incompressible fluid PDEs can also be interpreted as Hamiltonian systems associated with various energy functionals (Hamiltonians), and thus fit into a variational framework from the  Hamiltonian viewpoint. Specifically, in case of dimension $2$, such energies usually admit realization as functionals of a velocity field $u$ and an active scalar $\rho$ transported by $u$. Let $\Omega=\mathbb T^2$ be the physical domain of the fluid, typical examples include:
\begin{itemize}
    \item (i) The Boussinesq equation with energy given by ``kinetic plus gravitational potential'':
    \begin{equation*}
     E_{\text{Bou}}(u,\rho)=\int_{\Omega}(\frac{1}{2}\lvert u\rvert^2-\rho x_2) dx.   
    \end{equation*}
    \item (ii) The MHD equation with energy given by ``kinetic plus magnetic'':
    \begin{equation}
        E_{\text{MHD}}(u,\rho)=\int_{\Omega}\frac{1}{2}(\lvert u\rvert^2+\lvert\nabla\rho\vert^2) dx.
    \end{equation}
    In a more familiar form, the magnetic field and current read $B=\nabla^\perp\rho$ and  $J=\Delta\rho$ respectively.

    \item (iii) The inhomogeneous incompressible Euler equation with kinetic energy depending on the variable density:
    \begin{equation}
        E_{\text{IIE}}(u,\rho)=\int_{\Omega}\frac{1}{2}\lvert u\rvert^2\rho dx.
    \end{equation}
\end{itemize}

All of the above equations will be reduced to the incompressible Euler equation when $\rho=1$. Meanwhile, the global well-posedness of the above three equations is not known in $2$-d. Indeed, in contrast to the incompressible Euler equation, there are recent breakthroughs in finite time singularity formation of the Boussinesq equation in $2$-d, and we refer the interested readers to \cite{ElgindiPasqualotto2023}\cite{[EP25]}. The finite time blow up for MHD and IIE, to the best of the author's knowledge, remains a challenging open problem.

\subsection{Main Results}
The aim of this work is to establish estimates of the lifespan of the above three equations in close-to-Euler regime using an elementary bootstrap method, which is motivated by the  unified geometric formulation of incompressible fluids we described above, and is Lagrangian in nature. We informally state our main result here:
\begin{thm}\label{Lifespan of 2d Fluids}
   (Informal) Let $p>2$, assume that $\omega_0\in W^{1,p},\rho_0\in W^{2,p}$ and $\rho_0$ is close to $1$ in some suitable sense measured by a small parameter $\delta$. Taking $E=E_{\text{Bou}}$, $E= E_{\text{IIE}}$  respectively in the equation \eqref{Energy Vorticity}, then for sufficiently small $\delta$, the lifespan $T$ of the solution $(\omega,\rho)$ of \eqref{Energy Vorticity} admits a triple-log lower bound:
   \begin{equation}\label{Lifespan Lowerbound for Fluid PDEs}
      T\ge C_3\log [C_2\log (C_1\log (C_0\delta^{-1}))]\quad\text{ for some }C_0,C_1,C_2,C_3>0 
   \end{equation}
   For the case $E=E_{\text{MHD}}$, we need $\rho_0\in W^{4,p}$ and we have the same type of triple-log bound.
\end{thm}

The choice of the $W^{1,p}$ regularity class for the vorticity variable is crucial, as we require the control of $\lVert \nabla X \rVert_\infty$ to close the bootstrap argument. This necessity arises because our estimates for the linear transport equation do not benefit from the cancellation mechanisms inherent in nonlinearity of incompressible fluid PDEs (where high Sobolev norms are typically controlled by $\lVert \nabla u \rVert_\infty$). Whether such cancellation mechanism can be recovered within a purely Lagrangian framework (e.g. in view of Weber's formula in \cite{Co01}) is not clear.

The advantages of our Lagrangian approach are twofold: First, it stems from estimates of linear transport equation and standard bootstrap, which is clean and elementary. Second, it leads to the first result in continuation of the inhomogeneous Euler equation using just $L^\infty$ norm of vorticity, as a generalization of the Beale-Kato-Majda criterion.

The organization of the manuscript is as follows: We first introduce a novel energy-vorticity formulation of these hydrodynamic models, which is a variant of the Clebsch formulation of hydrodynamics. Then we combine the classical propagation of Sobolev regularity of transport equations with a simple abstract bootstrap argument to obtain estimates of lifespan of Boussinesq, MHD and IIE in close-to-Euler regime. The same approach also produce a quantitative BKM criteria for IIE using our vorticity formulation.

\subsection{Comparison with Existing Results}

 The longtime existence of two dimensional fluid models has been studied in recent years by several authors, we refer the interested readers to \cite{BaeLopesFilhoMazzucatoNussenzveigLopes2025}\cite{CobbFanelli2023}\cite{[DF11]}. Although we do not improve their results, we use a novel unified approach which is different from all existing works: in \cite{CobbFanelli2023} and \cite{[DF11]} the authors applied tools from paradifferential calculus to obtain the lifespan estimates in endpoint Besov spaces, in \cite{BaeLopesFilhoMazzucatoNussenzveigLopes2025} the authors used energy estimate and Gr\"onwall inequality, which are all Eulerian in nature. Meanwhile, our approach is to reduce the analysis to a perturbed system of differential inequalities using \textit{Lagrangian} representation of linear transport equations.

We also briefly recall some recent results in the continuation criteria of IIE here. In \cite{BaeLeeShin2020}, the authors showed that the blow up of IIE can be controlled by $L_t^1L_x^\infty$ norm of $\nabla u$ (which has four components in $2$-d) and in a very recent work \cite{Fanelli2026GeometricIIE}, a geometric criterion controlling blow up using $L_t^1L_x^\infty$ norm of $\partial_vu$ is obtained, where $v=\nabla^\perp\rho$.  Following the same elementary bootstrap approach we developed for lifespan estimates, we are able to derive a new quantitative BKM-type criterion for the inhomogeneous Euler equation based on our geometric formulation, which uses $L_t^1L_x^\infty$ norm of a single scalar quantity $\omega=\nabla\times (\rho u)$ to control the blow up and is closer to the BKM criterion \eqref{BKM L1Linfty Bound} in spirit.

\section*{Acknowledgements} The author thanks H.N.Lopes, M.Lopes, A.Mazzucato and W.Ozanski for stimulating discussions and valuable suggestions. This work was partially supported by the US National Science Foundation grants DMS-1909103 and DMS-2206453.

\section{Some Preparations}

\subsection{Conventions and Notations}

We will consider our fluid with velocity field $u$ and active density $\rho$ in the physical domain $\Omega=\mathbb T^2$ with periodic boundary condition.  The functional setting we will work with is the standard Sobolev space $W^{k,p}(\Omega)$. For notational convenience, we will use $W^{k,p}$ to indicate either $W^{k,p}(\Omega)$ or $W^{k,p}(\Omega;\mathbb R^2)$. Throughout the paper, we will assume $p>2$ to ensure $W^{2,p}\hookrightarrow W^{1,\infty}$, thanks to the Sobolev embedding. Moreover, under the same assumption, $W^{1,p}$ is an algebra.

We will denote by $X$ the flow map associated with a vector field $u\in W^{2,p}$, which is assumed to be divergence-free throughout this paper. Lipschitz regularity and divergence-free property of $u$ guarantee that $X$ is a well-defined curve of area-preserving diffeomorphisms. We denote by $A$ the back-to-label map: $A(t,x)=X_t^{-1}(x)$. 

We will always assume that our initial density $\rho_0$ satisfies
\begin{equation}\label{Unit mass of density}
 \int_\Omega \rho_0(x)dx=1   
\end{equation}
and
\begin{equation}\label{No Vaccum of rho}
  0< C_\rho^{-1}\le \rho_0(x)\le C_\rho ,\quad \text{for all }x\in D, 
\end{equation}
for some given $C_\rho>1$. Such a two-sided bound is preserved under rearrangement. We remark that our results will also hold in the whole-space case, with the normalization condition \eqref{Unit mass of density} dropped, while we keep the assumption \eqref{No Vaccum of rho}.

The letter $C$ is used for positive constants depending on the initial vorticity $\omega_0$, $p$ and $C_\rho$ unless specifically mentioned, and may change from line to line. We say that two numbers $A,B>0$ satisfy the relation $A\lesssim B$ if $A\le CB$. Two quantities $A,B$ are equivalent, denoted by $A\approx B$, if there exists a numerical constant $\alpha>1$ such that
\begin{equation*}
 \alpha^{-1}A\le B\le \alpha A.  
\end{equation*}

We will abuse the notation a little bit that the same numbers like $C_1$, $C_2$ will be repeatedly used in different contexts and different sections. Of course, they will not be related and will, hopefully, not cause any ambiguity.

\subsection{The Energy-Vorticity Formulation}

We recall the following Clebsch action functional for hydrodynamics (See \cite{HMR98}\cite{HSS09}):
\begin{equation*}
S(u,\rho,\chi)=\int_0^T E(u,\rho)-\langle \chi,\partial_t\rho+u\cdot\nabla\rho\rangle_2 dt  .  \end{equation*}
Here, $u$ is the divergence-free velocity field of the fluid, $\rho$ is a scalar density and $\chi$ is a scalar Lagrange multiplier enforcing $\rho$ to be transported by $u$.

Computing the first variation of the functional, we obtain the following system of critical point equations:
\begin{equation}\label{Critical Point of Clebsch}
\left\{
\begin{aligned}
&\partial_u E=\chi\nabla\rho+\nabla q,\\
&\partial_t \rho+u\cdot\nabla\rho=0,\\
& \partial_t \chi+u\cdot\nabla\chi=-\partial_\rho E.
\end{aligned}
\right.
\end{equation}
It is not hard to check that choose $E=E_{\text{Bou}}$, $E=E_{\text{IIE}}$ and $E=E_{\text{MHD}}$ respectively, we will recover the Clebsch variable representation of corresponding fluid PDEs.

Now, we define the vorticity $\omega=\nabla\times \partial_u E=\{\rho,\chi\}$. Combined with the transport equations:
\begin{equation*}
 \partial_t\rho+u\cdot\nabla\rho=0,\quad \partial_t\chi+\nabla\cdot(u\chi)=-\partial_\rho E   .
\end{equation*}
In Lagrangian coordinate, thanks to Duhamel's principle, we have the following formulae representing the solutions
\begin{equation*}
\rho=\rho_0\circ A_t,\quad \chi= \bigg(\chi_0-\int_0^t \partial_\rho E(X_\tau)d\tau \bigg)\circ A_t.  
\end{equation*}
Hence, in $2$-d,  the following Lagrangian formulae for vorticity holds:
\begin{equation*}
\omega=\nabla^\perp \rho\cdot\nabla\chi= [(\nabla  X\nabla^\perp\rho_0)\cdot(\nabla^*X)^{-1}(\nabla\chi_0-\int_0^t\nabla^* X_\tau\nabla\partial_\rho E(X_\tau)d\tau )]\circ A_t
\end{equation*}
\begin{equation*}
=\bigg(\omega_0-\nabla^\perp\rho_0\int_0^t \nabla(\partial_\rho E\circ X_\tau )d\tau\bigg)\circ A_t    
\end{equation*}
\begin{equation}\label{Lagrangian Formulae of Vorticity}
=\bigg(\omega_0-\nabla^\perp\rho_0\int_0^t \nabla^*X_\tau\nabla\partial_\rho E\circ X_\tau d\tau\bigg)\circ A_t,    
\end{equation}
which solves the following transport equation:
\begin{equation*}
 \partial_t\omega+u\cdot\nabla\omega=-\{\rho,\partial_\rho E\}  . 
\end{equation*}

Hence, we are ready to convert the equations \eqref{Critical Point of Clebsch} of $(u,\rho,\chi)$ into $(\omega,u,\rho)$ as follows:
\begin{equation}\label{Energy Vorticity}
\left\{
\begin{aligned}
&\partial_t \omega+u\cdot\nabla \omega=\{\partial_\rho E,\rho\},\\
&\partial_t \rho+u\cdot\nabla\rho=0,\\
& \nabla\times\partial_uE=\omega.
\end{aligned}
\right.
\end{equation}

The above system \eqref{Energy Vorticity} will be named as the energy-vorticity system. In this paper, we will consider the following two cases:
\begin{equation*}
E(u,\rho)=\frac{1}{2}\int_\Omega\lvert u\rvert^2 dx,\quad\text{ or }\quad   E(u,\rho)=\frac{1}{2}\int_\Omega\lvert u\rvert^2 \rho dx.
\end{equation*}

The recovery of $u$ from $\omega$ via the third equation is the standard Biot-Savart law for the first case:
\begin{equation*}
u=\mathcal K\omega:=\nabla^\perp\Delta^{-1}\omega.    
\end{equation*}
In this case, we will assume that $u_0$ is mean-zero, which will be propagated by the dynamics.

For inhomogeneous fluid with kinetic energy density depending on $\rho$, we will use a modified Biot-Savart law introduced in \cite{Pan2025LagrangianIIE}. More precisely, we have
\begin{equation}\label{Biot Savart for IIE on T2}
 u=\bP_\rho(\rho^{-1}\mathcal K\omega)= \rho^{-1}\mathcal K\omega-\rho^{-1}\nabla\Delta_{\rho}^{-1} \nabla\cdot( \rho^{-1}\mathcal K\omega),
\end{equation}
where we denote $\Delta_\rho=\nabla\cdot(\rho^{-1}\nabla)$. Here, our normalization will be that $\rho_0 u_0$ is mean-zero, corresponding to the conservation of linear momentum in inhomogeneous setting. To see the conservation of $\rho u$, we notice by \eqref{Biot Savart for IIE on T2}:
\begin{equation*}
 \int_\Omega \rho u dx =\int_\Omega \nabla^\perp\Delta^{-1}\omega dx-\int_\Omega \nabla\Delta_{\rho}^{-1} \nabla\cdot( \rho^{-1}\mathcal K\omega)dx=0.  
\end{equation*}

\subsection{Transport Lemmas and Inequalities}
In this subsection, we collect some useful analytical results. We first recall the following standard transport lemma which will be repeatedly used. For a more general version of propagation of Sobolev regularity in transport equation, we refer the readers to \cite{BCD11}.
\begin{lem}\label{Regularity of Transport}
 Let $f\in W^{2,p}$, assume  $X$ is the flow map associated with vector field $u\in W^{2,p}$, then we have:
 \begin{equation}\label{Transport Estimate}
   \lVert\nabla(f\circ X_t)\rVert_r\le \lVert\nabla X_t\rVert_\infty\lVert \nabla f\rVert_r,\quad \lVert f\circ X_t\rVert_{1,r}\le \lVert f\rVert+\lVert\nabla X_t\rVert_\infty \lVert \nabla f\rVert_r\quad\forall 1\le r\le \infty.
 \end{equation}
 Moreover, there exists constant $C$ such that the following exponential bound holds true:
 \begin{equation}\label{Exponential Bound for Deformation}
 \lVert\nabla X_t\rVert_\infty\lesssim \exp\bigg(C\int_0^t\lVert\nabla u\rVert_\infty d\tau\bigg),\quad \lVert \nabla X_t\rVert_{1,p}\lesssim \exp\bigg(C\int_0^t \lVert\nabla u\rVert_{1,p}d\tau\bigg),   
 \end{equation}
 and the above estimates \eqref{Transport Estimate} \eqref{Exponential Bound for Deformation} also hold true if we replace $X_t$ with $A_t=X_t^{-1}$.
\end{lem}

Thanks to \eqref{Lagrangian Formulae of Vorticity}, we are now ready to prove the following Lagrangian Duhamel estimate, which is not restricted to the perturbative regime:
\begin{lem}\label{W1p Estimate for Vorticity}
 Let $p>2$, let $\rho_0\in W^{2,p}$, $\omega_0\in W^{1,p}$. Moreover, we denote $\lVert \rho_0-1\rVert_{2,p}=\delta$ and define
\begin{equation}\label{Streching Variable Definition}
M=\exp\bigg(C\int_0^t \lVert\nabla u\rVert_\infty d\tau\bigg),\quad N= \exp\bigg(C\int_0^t \lVert\nabla u\rVert_{1,p} d\tau\bigg),
\end{equation}
and
\begin{equation*}
K_p=\lVert\nabla^2\partial_\rho E\rVert_{p},\quad K_0=\lVert\nabla\partial_\rho E\rVert_\infty .    
\end{equation*}
Where $C$ is the same as in \eqref{Exponential Bound for Deformation}. Then, $\omega$ satisfies the following apriori $W^{1,p}$ estimate:
 \begin{equation}\label{1,p regularity of omega}
 \lVert\omega(t,\cdot)\rVert_{{1,p}}\lesssim 1+M+\delta M\int_0^t (M+N)K_0+M^2K_p d\tau    ;
 \end{equation}
  \begin{equation}\label{2,p regularity of rho}
 \lVert\rho(t,\cdot)\rVert_{{2,p}}\lesssim 1+\delta(M+N+M^2) .
 \end{equation}
\end{lem}

\begin{proof}[Proof of Lemma~\ref{W1p Estimate for Vorticity}]
 The proof of \eqref{2,p regularity of rho} is standard and we omit it here. To demonstrate \eqref{1,p regularity of omega} In view of \eqref{Lagrangian Formulae of Vorticity}, straightforward calculation yields:
 \begin{equation*}
 \lVert\omega\rVert_{1,p}=
\lVert\omega\rVert_p+\lVert\nabla\omega \rVert_p    
 \end{equation*}
 \begin{equation*}
 \le   \lVert \omega_0\rVert_p+\lVert\nabla \rho_0\rVert_\infty \int_0^t \lVert\nabla A\rVert_\infty  \lVert\nabla\partial_\rho E\rVert_p d\tau ;
 \end{equation*}
 \begin{equation*}
 +\lVert\nabla A\rVert_\infty\lVert\nabla\omega_0\rVert_p+\lVert\nabla A\rVert_\infty \lVert\nabla^2\rho_0\rVert_p\int_0^t  \lVert\nabla A\rVert_\infty\lVert\nabla\partial_\rho E\rVert_\infty d\tau ;     
 \end{equation*}
\begin{equation*}
+\lVert\nabla A\rVert_\infty\lVert \nabla \rho_0\rVert_\infty\int_0^t \lVert\nabla^2 A\rVert_p\lVert\nabla\partial_\rho E\rVert_\infty d\tau+\lVert\nabla A\rVert_\infty\lVert\nabla\rho_0\rVert_\infty \int_0^t \lVert\nabla A\rVert_\infty\lVert\nabla X\rVert_\infty \lVert\nabla^2\partial_\rho E\rVert_p d\tau.     
\end{equation*}
Using the chord-arc bound \eqref{Exponential Bound for Deformation},  Sobolev embedding $W^{2,p}\hookrightarrow W^{1,\infty}$ and propagation of Sobolev regularity of transport equation \eqref{Transport Estimate}, we conclude
\begin{equation*}
\lVert\omega\rVert_{1,p}\lesssim 1+\delta\int_0^t MK_p d\tau+M+\delta M\int_0^t MK_0+\delta M\int_0^t NK_0d\tau+\delta M\int_0^t M^2K_p   
\end{equation*}
\begin{equation*}
 \lesssim 1+M+\delta M\int_0^t (M+N)K_0+M^2K_p d\tau  .  
\end{equation*}

\end{proof}

We will also frequently use the following logarithmic extrapolation inequality. The proof  can be found in \cite{MajdaBertozzi2002} and \cite{Constantin2017}. 
\begin{lem}\label{Logarithmic Sobolev}
 Let $d=2,3$ and $p>d$, let $u\in W^{2,p}$ be a divergence-free vector field and denote by $\omega=\nabla\times u$ the vorticity of $u$, then there exists a constant $C>0$ such that the following inequality holds:
\begin{equation}\label{Kato Inequality}
\lVert\nabla u\rVert_\infty\le C[1+\log(2+ \lVert\omega\rVert_{1,p})]\lVert\omega\rVert_\infty.
\end{equation}
\end{lem}

\section{Lifespan Estimates of Hydrodynamic Models}

In this section, we will derive lower bound estimates of lifespan of  three hydrodynamic models in $2$-d. More precisely, we will consider Boussinesq, MHD and IIE in near Euler regime. That is, the initial data $\rho_0$ is close to $1$ in an appropriate sense. 

Our strategy is to use a bootstrap procedure to handle the small parameter measuring the inhomogeneity. To illustrate the idea, we first quickly sketch a Lagrangian proof of Beale-Kato-Majda for $2$-d incompressible Euler. Since $E$ is independent of $\rho$, apply Lemma \ref{W1p Estimate for Vorticity} with $K_p=K_0=0$ and Lemma \ref{Logarithmic Sobolev}, we obtain:
\begin{equation*}
  \lVert\omega(t,\cdot)\rVert_{1,p}\lesssim 1+M,\quad  \frac{\dot M}{M}\lesssim [1+\log(2+ \lVert\omega\rVert_{1,p})]\lVert\omega(t,\cdot)\rVert_\infty
\end{equation*}
Therefore, we obtain 
\begin{equation*}
 \dot M\lesssim [1+M+(1+M)\log(1+M)]\lVert\omega(t,\cdot)\rVert_\infty.   
\end{equation*}
Let $\tilde M=1+\log(1+M)$, we have:
\begin{equation*}
 \frac{d}{dt}\tilde M\lesssim  \tilde M\lVert \omega(t,\cdot)\rVert_\infty,\quad \tilde M\le \exp\bigg(C\int_0^t \lVert\omega(\tau,\cdot)\rVert_\infty d\tau\bigg). 
\end{equation*}
Hence, we conclude
\begin{equation*}
  M(t)\lesssim   \exp\exp\bigg(C\int_0^t \lVert\omega(\tau,\cdot)\rVert_\infty d\tau\bigg).
\end{equation*}
 The double exponential growth of $M$ measures the Lagrangian stretching of flow $X$. Here, since the vorticity $\omega$ is transported, conservation of $\lVert\omega(t,\cdot)\rVert_\infty$ implies global well-posedness.

Now, for $2$-d fluid models with transported scalar $\rho$ in close-to-Euler regime, the source term $K_p$, $K_0$ appears in \eqref{1,p regularity of omega} with a small factor $\delta$. Bounding $K_0$ and $K_p$ usually requires higher norm estimate of $\rho$, which is controlled by $N$. On the other hand, we denote by
\begin{equation*}
 Q=\int_0^t (M+N)K_0+M^2 K_p d\tau,   
\end{equation*}
which is the Duhamel memory term appears in \eqref{1,p regularity of omega}. Then, assume that the following Calderon-Zygmund type estimate holds:
\begin{equation*}
 \lVert\nabla u\rVert_{1,p}\lesssim \lVert\omega\rVert_{1,p}.   
\end{equation*}
As a consequence, we have:
\begin{equation*}
 \dot N=\lVert\nabla u\rVert_{1,p}N\lesssim \lVert  \omega\rVert_{1,p} N\lesssim (1+M+\delta MQ)N.
\end{equation*}
Hence, suppose that $\delta Q(t)\le 1$ in a sufficiently small time interval $[0, T_\delta]$, then we may control $N$ via a triple exponential. Such a bound feeds back into the growth estimate of $Q$, which also turns out to be triple exponential in the models we study and suggests a triple-log bound for $T_\delta$. We rigorously realize the above sketch of idea via a bootstrap argument. 

\subsection{A Growth Lemma and Triple-Log Lifespan Estimate}

The main goal of this subsection is to carry out a hierarchical growth lemma and an abstract bootstrap closure lemma. These lemmas are very elementary, but they successfully capture the growth of norm and lifespan estimates in various fluid models in a unified way, which turns out to be powerful.

Let us first state the abstract bootstrap lemma (cf. \cite{Tao2006NonlinearDispersive}), which we will use throughout this section.

\begin{lem}\label{Bootstrap Lemma}
    Let $I$ be an interval of time. Suppose we have one hypothesis $\mathscr H(t)$ and one outcome $\mathscr O(t)$ defined for each $t\in I$, then $\mathscr O(t)$ holds true for all $t\in I$ provided the validity of the following statements:
    \begin{itemize}
        \item [(I)] Suppose $\mathscr H(t)$ is true at $t\in I$, then $\mathscr O(t)$ is true.

        \item [(II)] Suppose $\mathscr O(t)$ holds for $t\in I$, then there exists a neighborhood $I_t\owns t$ such that $\mathscr H(t^\prime)$ is true for all $t^\prime\in I_t$.

        \item [(III)] There exists a $t_*\in I$ such that $\mathscr H(t_*)$ is true.

        \item [(IV)] Assume $\{t_n\}_{n=1}^\infty$ is a sequence of time converging to $\tilde t$ such that $\mathscr O(t_j)$ is true for all $j$, then $\mathscr O(\tilde t)$ is true.
    \end{itemize}
\end{lem}

\begin{proof}[Proof of Lemma~\ref{Bootstrap Lemma}]

Let $I_{\mathscr O}\subset I$ be the interval where $\mathscr O(t)$ holds true. Condition (IV) guarantees that $I_{\mathscr O}$ is closed, (I) and (III) guarantees that $I_{\mathscr O}$ is nonempty, while (I) and (II) guarantees that $I_{\mathscr O}$ is open. Hence, we conclude $I_{\mathscr O}=I$. 
    
\end{proof}

We now state our hierarchical growth lemma. 
\begin{lem}\label{Hierarchical Growth of Norms for Bou and IIE}
  Let $M,N,Q: [0,T]\to \mathbb R_+$ be absolutely continuous non-decreasing functions, such that
  \begin{equation*}
  M(0)=1, \quad N(0)=1,\quad Q(0)=0.    
  \end{equation*}
  Suppose that there exists an interval $[0,T_*]$ and a constant $C>e$ such that the following holds on $[0,T_*]$:
  \begin{equation}\label{Differential inequality of MYZ}
   \dot M\le CM(1+\log(1+\frac{\dot N}{N})),\quad \dot N\le CN(1+M),  
  \end{equation}
  and
  \begin{equation}\label{Differential inequality of Duhamel Term}
    \dot Q\le C(\dot M+N\dot M/M)+M\dot M\dot N/N +C(M+N). 
  \end{equation}
  Then, there exists positive constants $\Lambda_1,\Lambda_2,\Lambda_3,\Lambda_4$ such that
  \begin{equation}\label{Growth of controlling norms}
  M(t)\le \exp(e^{\Lambda_1 t}),\quad N(t)\le \exp(\Lambda_2\exp(e^{\Lambda_1t})),\quad Q(t)\le \Lambda_4 \exp(\Lambda_3\exp( e^{\Lambda_1t})) .  
  \end{equation}
  
\end{lem}
\begin{proof}[Proof of Lemma~\ref{Hierarchical Growth of Norms for Bou and IIE}]
 The proof is straightforward and clean. By \eqref{Differential inequality of MYZ}, we have for $\tilde C:=1+\log(1+C)$:
 \begin{equation}\label{Double EXP ODE for M}
  \dot M\le \tilde C(1+M)(1+\log (1+M)) . 
 \end{equation}
We define $\mathfrak M:=\log(1+M)+1$, then we have $\mathfrak M(0)=1$. By \eqref{Double EXP ODE for M}:
\begin{equation*}
 \dot{\mathfrak M}\le \tilde C\mathfrak M\quad\Longrightarrow\quad \mathfrak M(t)\le \exp(\tilde Ct)  \quad\Longrightarrow \quad M(t)\le e^{-1}\exp(e^{\tilde Ct})-1\le\exp(e^{\tilde Ct}).
\end{equation*}
Which automatically gives:
 \begin{equation*}
   N\le \exp\bigg( C\int_0^t[1+M(\tau)]d\tau\bigg)\le  \exp[C\exp(\exp(\tilde Ct))].
 \end{equation*}
 Now, it suffices to estimate $Q$. Notice:
 \begin{equation*}
 Q(t)\le C\int_0^t(\dot M+N\dot M/M) d\tau+\int_0^t M\dot M\dot N/N d\tau+C\int_0^t(M+N) d\tau   
 \end{equation*}
 \begin{equation*}
  \le    CM(t)+CN(t) \log M(t)+ CM(t)\int_0^t \dot (1+M(\tau))\dot M(\tau)d\tau+C\int_0^t(M+N)d\tau
 \end{equation*}
 \begin{equation*}
  \le  C\exp(\exp(\tilde Ct))+C\exp[2C\exp(\exp(\tilde Ct))] + \exp(3\exp(\tilde Ct))
 \end{equation*}
 \begin{equation*}
   +C\exp(\exp(\tilde Ct))+C\exp[C\exp(\exp(\tilde Ct))]  
 \end{equation*}
 \begin{equation*}
  \le    5C\exp[2C\exp(\exp(\tilde Ct))].
 \end{equation*}
 Here, we repeatedly used the increasing assumption on $M$ and $N$, also we used the bound $C>e$. Hence, the desired claim \eqref{Growth of controlling norms} is demonstrated, with
 \begin{equation*}
   \Lambda_1=\tilde C,\quad\Lambda_2=C,\quad \Lambda_3=2C\quad\text{and}\quad\Lambda_4=5C.  
 \end{equation*}
\end{proof}

The above growth lemma actually encodes the triple exponential growth of $\lVert\nabla X\rVert_{1,p}$ ($N$ variable in the lemma) for $2$-dimensional Boussinesq and inhomogeneous Euler equation for short time. To quantify such a statement, we introduce the following bootstrap closure lemma.
\begin{lem}\label{Bootstrap Closure Lemma}
 Let $I=[0,T]$ and $\delta>0$, assume that $\{\mathcal F_j\}_{j=1}^m: I\to [0,+\infty)$ are continuous increasing functions with $\mathcal F_j(0)\le 1$ holding for all $1\le j\le m$, and we assume $\{\kappa_j\}_{j=1}^m$ is a sequence of positive real numbers. Let
 \begin{equation*}
  \mathcal E(t):=\exp(C_1\exp(C_2 e^{C_3t})),\quad C_1,C_2,C_3>0.   
 \end{equation*}
 Assume that there exists $C_1,C_2,C_3>0$ and sequence $\{\zeta_j\}_{j=1}^m \subset [1,+\infty)$ such that the following implication holds: For some $0<t\le T$:
 \begin{equation}\label{Bootstrap Implication}
\kappa_j\delta \mathcal F_j(t)\le 1,\quad\forall 1\le j\le m \quad\Longrightarrow\quad   \mathcal F_j(s)\le \zeta_j\mathcal E(s), \quad\forall 1\le j\le m, s\in [0,t].
 \end{equation}

 Then, for some small threshold 
 \begin{equation}\label{Delta0 Definition}
  \delta_0=C_0\exp(-C_1 e^{2C_2})  
 \end{equation}
 and $\delta<\delta_0$, we have: 
 \begin{equation*}
  \kappa_j\delta \mathcal F_j(t)\le 1  \quad\text{for all }1\le j\le m , t\in[0,T_\delta], 
 \end{equation*}
 where
 \begin{equation}\label{Triple Log Lifespan}
  T_\delta=C_3^{-1}\log[C_2^{-1}\log(C_1^{-1}\log(C_0\delta^{-1}))],\quad C_0:=\frac{99}{100\max_{1\le j\le m}  \kappa_j \zeta_j}.
 \end{equation}
 As a matter of fact, we have:
 \begin{equation}\label{Finiteness on Bootstrap interval}
   \mathcal F_j(t)  \le \zeta_j C_0\delta^{-1}\quad\text{for all }1\le j\le m, t\in[0,T_\delta].
 \end{equation}
\end{lem}
\begin{proof}[Proof of Lemma~\ref{Bootstrap Closure Lemma}]
  We consider the following natural bootstrap assumption $\mathscr H(t)$: $\mathscr H(t)$ is said to hold if we have:
  \begin{equation*}
   \kappa_j\delta \mathcal F_j(t)\le 1,\quad\forall 1\le j\le m .   
  \end{equation*}
 The associated bootstrap outcome $\mathscr O(t)$ is selected to be
 \begin{equation*}
    \mathcal F_j(s)\le \zeta_j\mathcal E(s), \quad\forall 1\le j\le m, s\in [0,t].  
 \end{equation*}
 We now check the four conditions in Lemma \ref{Bootstrap Lemma}. $\mathscr H(t)$ implies $\mathscr O(t)$ is our assumption. $\mathscr H(0)$ holds since 
 \begin{equation*}
  \kappa_j\delta<\kappa_j \delta_0=\frac{99\kappa_j}{100\max_{1\le j\le m}\kappa_j\zeta_j}\exp(-C_1e^{2C_2})\le \frac{99\kappa_j}{100 \max_{1\le j\le m}\kappa_j}\le 1,\quad \text{for all }1\le j\le m.   
 \end{equation*}
 The closedness condition (IV) holds trivially for $\mathscr O$, and it suffices to check the feedback condition (II).

 Now, for any $t\in [0,T_\delta]$, we first observe that $T_\delta$ is decreasing in $\delta$ provided $T_\delta$ is well-defined. In particular, $T_{\delta}\ge T_{\delta_0}$. Now we compute:
 \begin{equation*}
 T_{\delta_0}=C_3^{-1}\log[C_2^{-1}\log(C_1^{-1}\log(C_0\delta_0^{-1}))] =C_3^{-1} \log 2>0.  
 \end{equation*}
 On the other hand, using the increasing assumption on $\mathcal F_j$, we have for all $t\in [0,T_\delta]$:
 \begin{equation}
 \delta \kappa_j\mathcal F_j(t) \le \delta \kappa_j\mathcal F_j(T_{\delta})\le \delta\kappa_j\zeta_j \mathcal E(T_\delta)=\delta\kappa_j\zeta_j \cdot C_0\delta^{-1} \le \frac{99}{100}  
 \end{equation}
Hence, there must exists $I_t\owns t$ such that $\mathscr H(t)$ holds true on $I_t$. In particular, this bootstrap closes on $[0,T_\delta]$ with $T_\delta\ge T_{\delta_0}= C_3^{-1}\log 2>0$.
\end{proof}

\begin{rmk}
  Our approach fails in $3$-d cases due to the vortex stretching effect. From a Lagrangian viewpoint, the $L^p$ norm of vorticity is not conserved, but stretched by the Jacobian matrix $\nabla X$ of the flow map instead. As a consequence, we will not be able to control $N$ exponentially in $M$. Instead, the presence of quadratic term in non-perturbative part of the differential inequality will destroy the bootstrap argument. 
  
  On the other hand, one should not expect any asymptotic global existence result in $3$-d for Boussinesq or Inhomogeneous Euler, since the global existence of incompressible Euler in $\mathbb T^3$ is not known. So a natural direction in this setting should be trying to develop an asymptotic lifespan comparison type result.
\end{rmk}

\subsection{Boussinesq Case}

We now establish the long-time existence of Boussinesq equation in $2$-d in the close-to-Euler regime. More precisely, we fix $p>2$ and define $\delta:=\lVert \rho_0-1\rVert_{2,p}$ which measures the inhomogeneity of initial data.

The main theorem of this subsection is:
\begin{thm}\label{Lifespan of Boussinesq} There exist constants $\delta_0$ and $C_0^*,C_1^*, C_2^*, C_3^*>0$ possibly depending on the initial data $\omega_0$, such that: For any small positive number $\delta<\delta_0$, the quantity 
\begin{equation*}
\mathsf E(t):=\lVert\omega\rVert_{1,p}+\lVert\rho\rVert_{2,p}   
\end{equation*} 
remains bounded on the interval $[0,T_\delta]$, where
\begin{equation}\label{Triple Log Lower bound for Lifespan}
 T_\delta=C_3^*\log [C_2^*\log (C_1^*\log( C_0^*\delta^{-1}))]   .
\end{equation}
\end{thm}

We will use an elementary Lagrangian approach to prove the result. The idea is to establish \textit{a priori} estimates of the $W^{1,\infty}$ and $W^{2,p}$ norm of the flow map $X$ of $u$, which leads to a system of perturbed differential inequalities and fits Lemma \ref{Hierarchical Growth of Norms for Bou and IIE}. The rest of the work boils down to check the conditions in Lemma \ref{Bootstrap Closure Lemma}.

\begin{proof}[Proof of Theorem~\ref{Lifespan of Boussinesq}]

We divide the proof into several steps.
\begin{itemize}
    \item \textbf{Step 1.} We first establish a system of differential inequalities given by the apriori estimates. Notice that in the Boussinesq case, we have 
    \begin{equation*}
     E(u,\rho)=\frac{1}{2}\lVert u\rVert_2^2 -\langle \rho,x_2\rangle,\quad \nabla\partial_\rho E=-e_2, \nabla^2\partial_\rho E=0   .
    \end{equation*}
    Hence, we obtain $K_0= 1$, $K_p=0$ in this case, and by Lemma~\ref{W1p Estimate for Vorticity} we have for some $C> 0$:
    \begin{equation*}
  \lVert\omega\rVert_{1,p}\le C+CM+CM\delta\int_0^ t (M+N)K_0 d\tau   .  
    \end{equation*}
    Using the Calderon-Zygmund estimate $\lVert\nabla u\rVert_{1,p}\le C\lVert\omega\rVert_{1,p}$, we conclude:
    \begin{equation*}
    \dot N\le CN (1+M+M\delta Y),\quad Y=\int_0^t (M+N) d\tau  . 
    \end{equation*}
    Meanwhile, by Lemma~\ref{Logarithmic Sobolev}, we have
    \begin{equation*}
     \dot M\le CM(1+\log(2+ \frac{\dot N}{N})) \lVert\omega\rVert_\infty .  
    \end{equation*}
    Plus, we have the following $L^\infty$ bound for $\omega$:
    \begin{equation*}
    \lVert\omega\rVert_\infty\le   C+C\delta Z,\quad Z=\int_0^t Md\tau  .
    \end{equation*}
    Hence, we end up with the following perturbed system of differential inequalities:
    \begin{equation}\label{Boussinesq Differential Inequalities}
\left\{
\begin{aligned}
&\dot M\le CM[1+\log(2+\frac{\dot N}{N})] (1+\delta Z),\\
&\dot N\le CN (1+M+M\delta Y),\\
&  \dot Y=M+N,\quad \dot Z=M,\quad Y_0=Z_0=0.
\end{aligned}
\right.
\end{equation}
So far, we have not used any smallness assumption on $\delta$, and the above system of  differential inequality with memory holds unconditionally. Moreover, if we assume that $\delta=0$, we will obtain the double exponential growth of $M$ and triple exponential growth of $N$ as a consequence. 

\item \textbf{Step 2.} We now establish our bootstrap procedure towards the proof. We propose the following bootstrap assumption $\mathscr H(t)$: We say $\mathscr H(t)$ holds, if 
\begin{equation}\label{Bootstrap Assumption}
 \delta Z(t)\le 1,\quad \delta Y(t)\le 1   
\end{equation}
 Since $M,N>0$, the above assumption reduces to $\delta Y(t)\le 1$.

Now, thanks to the increasing property of $Y$, we conclude that $\mathscr H(t)$  holds true for all $t\in [0,T_*]$, provided $\mathscr H(T_*)$ holds true. Therefore, $(M,N,Y)$ satisfy a relaxed system of inequality on $[0,T]$:
 \begin{equation*}
\left\{
\begin{aligned}
&\dot M\le 2CM(1+\log(2+\frac{\dot N}{N})) \\
&\dot N\le 2CN (1+M)\\
& \dot Y\le M+N,\quad Y(0)=0.
\end{aligned}
\right.
\end{equation*}
With out lose of generality, we set $C>e$. 

Apply Lemma \ref{Hierarchical Growth of Norms for Bou and IIE} with $Q=Y$, we conclude that there exists $\Lambda_1,\Lambda_2,\Lambda_3,\Lambda_4>0$ such that for all $t\in [0,T_\delta]$:
\begin{equation*}
M(t)\le \exp(e^{\Lambda_1t}),\quad N(t)\le \exp(\Lambda_2\exp(e^{\Lambda_1 t})),\quad Y(t)\le \Lambda_4\exp(\Lambda_3\exp(e^{\Lambda_1 t})).
\end{equation*}
The above growth bound is chosen as our bootstrap outcome $\mathscr O(t)$.

We already showed that $\mathscr H(t)$ implies $\mathscr O(t)$. Meanwhile, by absolute continuity of $Y$ and $Y(0)=0$, the condition (III) is readily verified. The closedness condition (IV) is obvious and it suffices to verify (II).

\textbf{Step 3.} We now complete the bootstrap procedure. As we already guarantee $\mathscr H(t)$ implies $\mathscr O(t)$ and $\mathscr H(t)$ implies for any $\mathscr H(s)$ with $s\in [0,t]$, we can apply Lemma \ref{Bootstrap Closure Lemma} with $m=1$, $\mathcal F_1(t)=Y(t)$, $\kappa_1=1$ and $\zeta_1=\Lambda_4$. Hence:
\begin{equation*}
C_0=\frac{99}{100 \Lambda_4},\quad \delta_0=C_0\exp(-C_3 e^{2}),\quad T_\delta=\Lambda_1^{-1}\log\log\Big[\Lambda_3\Big(\log \frac{99}{100\Lambda_4\delta}\Big)\Big].
\end{equation*}

 and for all $t\in [0,T_\delta]$, by \eqref{Finiteness on Bootstrap interval} we have
\begin{equation*}
\mathsf E(t)\le  C+CY+C\delta Y^2\lesssim 1+\delta^{-1}.
\end{equation*}
  Therefore, we conclude that $\mathsf E$ is finite on $[0,T_\delta]$ and the solution can be continued past time $T_\delta$, with desire triple-log lower bound \eqref{Triple Log Lower bound for Lifespan} of $T_\delta$.
\end{itemize}

\end{proof}

\subsection{The Two-Dimensional IIE}
Now we study the lifespan of inhomogeneous Euler equation in the near-Euler regime. We recall that IIE is known to be locally well-posed in H\"older and Hilbert Sobolev spaces (see \cite{BeiraoDaVeigaValli1978}\cite{BeiraoDaVeigaValli1980I}\cite{BeiraoDaVeigaValli1980II}\cite{VALLI198843}), and we refer the reader to \cite{Da06}\cite{[Da10]}\cite{[DF11]} for various well-posedness results in $L^p$-based Sobolev spaces and endpoint Besov spaces.

For convenience of our representation, we rephrase our setting as:
\begin{equation*}
 \mu_0=\rho_0^{-1}=1+\delta\theta_0,\quad \lVert\theta_0\rVert_{2,p}=1.   
\end{equation*}
By the transport nature of $\rho$, we have
\begin{equation*}
 \mu(t,x)=\rho^{-1}(t,x)-1=\delta\theta(t,x)=\delta\theta_0(X_t^{-1}(x)),\quad \dot X=u(X).   
\end{equation*}

Plugging $E=E_{\text{IIE}}$ in \eqref{Energy Vorticity}, we derive the energy-vorticity formulation of IIE:
\begin{equation}\label{IIE vorticity}
\left\{
\begin{aligned}
&\partial_t \omega+u\cdot\nabla \omega=\{\frac{1}{2}\lvert u\rvert^2,\rho\};\\
&\partial_t \rho+u\cdot\nabla\rho=0;\\
& u=\rho^{-1}\mathcal K\omega-\rho^{-1}\nabla\Delta_{\rho}^{-1}\nabla\cdot(\rho^{-1}\mathcal K\omega),\quad \mathcal K=\nabla^\perp\Delta^{-1}.
\end{aligned}
\right.
\end{equation}

Now, thanks to the perturbative representation of $\rho$, we use the div-free property of $\mathcal K\omega$ to derive
\begin{equation*}
u=\rho^{-1}\mathcal K\omega-\rho^{-1}\nabla\Delta_{\rho}^{-1}\nabla\cdot(\rho^{-1}\mathcal K\omega)   
\end{equation*}
\begin{equation*}
=  \mathcal K\omega-\mu\nabla\Delta_{\rho}^{-1}\nabla\cdot[(\mu-1)\mathcal K\omega]+(\mu-1)\mathcal K\omega  
\end{equation*}
\begin{equation}\label{closed form of u}
 = \mathcal K\omega-\delta(1+\delta \theta)\nabla\Delta_{\rho}^{-1}\nabla\cdot(\theta\mathcal K\omega)+\delta\theta \mathcal K\omega
\end{equation}

Let the mean-zero momentum variable $\eta=\rho u$ be such that $\eta=\nabla^\perp\psi+\nabla q$, we have:
\begin{equation}\label{Equation for Grad part of Momentum}
\Delta_\rho q=-\nabla^\perp\psi\cdot\nabla\mu=-\delta\nabla^\perp\psi\cdot\nabla \theta ,\quad \omega=\Delta\psi . 
\end{equation}
Notice that:
\begin{equation*}
\nabla^\perp\psi\cdot\nabla \theta =\nabla\cdot(\theta\nabla^\perp\psi)=\nabla\cdot(\theta\mathcal K\omega).   
\end{equation*}
Hence, by \eqref{Equation for Grad part of Momentum}, we can rewrite \eqref{closed form of u} as:
\begin{equation}\label{Compact perturbative formula for u}
 u=  \mathcal K\omega +\delta\theta \mathcal K\omega-(1+\delta\theta)\nabla q,\quad \Delta q+\delta\nabla\times( \theta \nabla^\perp q)=-\delta\nabla\cdot(\theta\mathcal K\omega).
\end{equation}

We now establish several estimates which will be useful in the sequel. The first lemma is an elementary estimate on $\rho_0$.
\begin{lem}\label{Rho 0 Estimate}
Assume $\delta<1$, then we have
\begin{equation}
 \lVert\nabla\rho_0\rVert_\infty\lesssim \delta,\quad \lVert \nabla^2\rho_0\rVert_p\lesssim\delta.
\end{equation}
\end{lem}
\begin{proof}[Proof of Lemma~\ref{Rho 0 Estimate}]
 Since $\rho_0=(1/x)\circ \mu_0=(1+\delta\theta_0)^{-1}$, we have
 \begin{equation*}
 \rho_0=1-\frac{\delta\theta_0}{1+\delta\theta_0}.   
 \end{equation*}
 Differentiating $\rho_0$ gives:
 \begin{equation*}
  \lVert\nabla\rho_0\rVert_\infty\le \delta\bigg\lVert \frac{\nabla\theta_0}{1+\delta\theta_0}\bigg\rVert_\infty + \delta^2\bigg\lVert \frac{\theta_0\nabla\theta_0}{(1+\delta\theta_0)^2}\bigg\rVert_\infty 
 \end{equation*}
 \begin{equation*}
 \le \delta C_\rho\lVert \theta_0\rVert_\infty+\delta^2 C_\rho^2\lVert\theta_0\rVert_\infty\lVert\nabla\theta_0\rVert_\infty\lesssim \delta.    
 \end{equation*}
 Taking one more derivative, we obtain
 \begin{equation*}
 \lVert\nabla^2\rho_0\rVert_p \le  \delta\bigg\lVert \frac{\nabla^2\theta_0}{1+\delta\theta_0}\bigg\rVert_{p} +2\delta^2\bigg\lVert \frac{\nabla\theta_0\otimes\nabla\theta_0}{(1+\delta\theta_0)^2}\bigg\rVert_p +\delta^2\bigg\lVert \frac{\nabla\theta_0\otimes\nabla\theta_0+\theta_0\nabla^2\theta_0}{(1+\delta\theta_0)^2}\bigg\rVert_{p}+2\delta^3\bigg\lVert \frac{\theta_0\nabla\theta_0\otimes\nabla\theta_0}{(1+\delta\theta_0)^3}\bigg\rVert_{p}
 \end{equation*}
 \begin{equation*}
  \lesssim   \delta C_\rho\lVert\theta_0\rVert_{2,p}+2C_\rho^2\delta^2 \lVert\theta_0\rVert_{2,p}^2+C_\rho^2\delta^2(\lVert\theta_0\rVert_{2,p}^2+\lVert\theta_0\rVert_{2,p}^3)+2\delta^3 C_\rho^3 \lVert\theta_0\rVert_{2,p}^3\lesssim\delta.
 \end{equation*}
\end{proof}
The second lemma provides useful apriori estimates for $u$.
\begin{lem}\label{Estimates of u}
Let $M(t),N(t)\ge 1$ be defined as in \eqref{Streching Variable Definition}, then
 there is a numerical constant $C\ge 1$ which satisfies the following: For any fixed  $\delta<C^{-1}$, there exists a corresponding time $T_\delta^*>0$ such that for all $t\in [0,T_\delta]$ the following estimates hold:
\begin{equation}\label{W2p estimate for u}
 \lVert u\rVert_{2,p}\le  \frac{ (1+\delta N)\lVert\omega\rVert_{1,p}}{1-C\delta N}
\end{equation}
\begin{equation}\label{Linfty estimate for nabla u}
 \lVert\nabla u\rVert_\infty\lesssim   \Big(\delta M\lVert\omega\rVert_\infty+\frac{ \lVert\omega\rVert_\infty}{1-C\delta M}\Big)\Big[1+\log 
 \Big(2+\frac{ (1+\delta N)\lVert\omega\rVert_{1,p}}{1-C\delta N}\Big)\Big].
\end{equation}
\end{lem}

\begin{proof}[Proof of Lemma~\ref{Estimates of u}]

We first consider the elliptic equation for $q$ in \eqref{Compact perturbative formula for u}. Let $v=\nabla^\perp q$, we apply Biot-Savart operator to the equation of $q$ to obtain
\begin{equation*}
 v+ \delta\bP(\theta v)= -\delta \bP(\theta\nabla\Delta^{-1}\omega). 
\end{equation*}
Hence, for any $k\ge 1$ and $p>2$, there exists a constant $C_1>0$ such that:
\begin{equation*}
\lVert v\rVert_{k,p}\le  \delta \lVert \bP(\theta v)\rVert_{k,p} +\delta  \lVert\bP(\theta\nabla\Delta^{-1}\omega)\rVert_{k,p}
\end{equation*}
\begin{equation*}
 \le  C_1 \delta  (\lVert \theta\rVert_{k,p}\lVert v\rVert_{k,p}+\lVert \theta\rVert_{k,p}\lVert \omega\rVert_{k-1,p}),
\end{equation*}
where we use the fact that $W^{k,p}$ is an algebra and $\bP$ is bounded on $W^{k,p}$. Therefore, for $\delta$ sufficiently small such that $C_1\delta\lVert\theta\rVert_{k,p}<1$, we have:
\begin{equation*}
\lVert v\rVert_{k,p}\le \frac{C_1\delta \lVert \theta\rVert_{k,p}\lVert \omega\rVert_{k-1,p}}{1-C_1\delta \lVert \theta\rVert_{k,p}}    .
\end{equation*}
Going back to the equation for $u$ in \eqref{Compact perturbative formula for u}, we have
\begin{equation*}
 \lVert u\rVert_{k,p}\le  (1+\delta \lVert\theta\rVert_{k,p})(\lVert\omega\rVert_{k-1,p}+\lVert v\rVert_{k,p})
\end{equation*}
\begin{equation}\label{u Wk,p Estimate}
 \le (1+\delta \lVert\theta\rVert_{k,p})\lVert\omega\rVert_{k-1,p}+(1+\delta \lVert\theta\rVert_{k,p})\frac{C_1\delta \lVert \theta\rVert_{k,p}\lVert \omega\rVert_{k-1,p}}{1-C_1\delta \lVert \theta\rVert_{k,p}}.   
\end{equation}
Meanwhile, taking curl on the $u$-equation in \eqref{Compact perturbative formula for u}, by Sobolev embedding $W^{1,p}\hookrightarrow L^\infty$ and compactness of $D$, there exists a constant $C_0>0$ such that
\begin{equation*}
 \lVert \nabla\times u\rVert_\infty\le (1+\delta\lVert\theta \rVert_{\infty})\lVert \omega\rVert_\infty+ C_0\delta\lVert \omega\rVert_{p}\lVert\theta\rVert_\infty +C\delta \lVert v\rVert_{1,p}\lVert \nabla\theta\rVert_\infty   
\end{equation*}
\begin{equation}\label{Linfty estimate of curl u}
 \le    (1+2\delta\lVert\theta \rVert_{\infty})\lVert \omega\rVert_\infty+ C_0\delta \frac{C_1\delta \lVert \theta\rVert_{1,p}\lVert \omega\rVert_{\infty}}{1-C_1\delta \lVert \theta\rVert_{1,p}}
\end{equation}
Now, classical Cauchy-Lipschitz theory implies that
\begin{equation}\label{transport estimate of theta}
 \lVert \theta(t,\cdot)\rVert_{2,p}\le N(t)\lVert \theta_0\rVert_{2,p}=N(t),\quad \lVert\theta\rVert_{1,p}\le M(t),\quad \lVert\nabla\theta\rVert_\infty\le M(t)\lVert\nabla\theta_0\rVert_\infty\le M(t).   
\end{equation}
Now, we set $t$ smaller if needed to guarantee $C_1\delta M(t)<1$, $C_1\delta N(t)<1$. Thanks to \eqref{transport estimate of theta}, we may control various norms of $\theta$ using $M$, $N$. Combining \eqref{u Wk,p Estimate} and \eqref{Linfty estimate of curl u} with \eqref{transport estimate of theta}, we have:
\begin{equation}\label{ u final W2,p estimate}
\lVert u\rVert_{2,p}\le \lVert\omega\rVert_{1,p}+\delta N\lVert\omega\rVert_{1,p}+(1+\delta N)\frac{C_1\delta N\lVert\omega\rVert_{1,p}}{1-C_1\delta N}=\frac{(1+\delta N)\lVert\omega\rVert
_{1,p}}{1-C_1\delta N}.   
\end{equation}
Moreover, we assume $\delta\le C_0^{-1}$ and obtain:
 \begin{equation}\label{ curl u final Linfty estimate}
\lVert \nabla\times u\rVert_\infty\le (1+2\delta M)\lVert\omega\rVert_\infty+C_0\delta\frac{C_1\delta M \lVert\omega\rVert_\infty}{1-C_1\delta M}\le 2\delta M \lVert\omega\rVert_\infty+\frac{\lVert\omega\rVert
_\infty}{1-C_1\delta M}  
\end{equation}

Meanwhile, the logarithmic Sobolev inequality \eqref{Logarithmic Sobolev} yields existence of constant $C_2$ such that the following holds:
\begin{equation*}
 \lVert \nabla u\rVert_\infty\lesssim \lVert\nabla\times u\rVert_\infty[1+\log (2+\lVert\nabla\times u\rVert_{1,p})]\le C_2   \lVert\nabla\times u\rVert_\infty[1+\log (2+\lVert u\rVert_{2,p})]
\end{equation*}
\begin{equation*}
 \le  \Big[2C_2\delta M\lVert\omega\rVert_\infty+\frac{C_2\lVert\omega\rVert_\infty}{1-C_1\delta M}\Big][1+\log (2+\lVert u\rVert_{2,p})].
\end{equation*}

Hence, we have
\begin{equation*}
 \frac{\dot M}{M}\le   \Big[2C_2\delta M\lVert\omega\rVert_\infty+\frac{C_2 \lVert\omega\rVert_\infty}{1-C_1\delta M}\Big]\Big(1+\log 
 \Big(2+\frac{\dot N}{N}\Big)\Big).
\end{equation*}
On the other hand, by \eqref{W2p estimate for u} we have:
\begin{equation}\label{W2p growth of u}
\frac{\dot N}{N}\le \frac{(1+\delta N)\lVert\omega\rVert_{1,p}}{1-C_1\delta N}.
\end{equation}
Plugging \eqref{W2p growth of u} back to \eqref{Linfty estimate for nabla u}, the desired result follows from choosing $C$ and $T_\delta$ as
\begin{equation*}
C=\max\{1,C_0,C_1,2C_2\},\quad T_\delta^*=\min\big\{t\ge 0: M(t)\ge \frac{3}{4C\delta}, N(t)\ge \frac{3}{4C\delta}\big\},     
\end{equation*}
which is legitimate as $M,N$ are absolutely continuous increasing functions with $M(0)=N(0)=1$.
\end{proof}

We now state and prove the first theorem concerning lifespan of IIE in perturbative regime.  We will make use of the controlling constant $C$ appeared above.
\begin{thm}\label{Lifespan of IIE}
There exist positive constants $\delta_0<1$ and $C_0^*,C_1^*, C_2^*,C_3^*$, possibly depending on the data $\omega_0$, such that for any $\delta<\delta_0$, $\mathsf E(t):=\lVert\omega\rVert_{1,p}+\lVert\rho\rVert_{2,p}$ remains bounded on the interval $[0,T_\delta]$, where
\begin{equation*}
 T_\delta=C_3^*\log[C_2^*\log( C_1^*\log (C_0^*\delta^{-1}))]. 
\end{equation*}
\end{thm}
\begin{proof}[Proof of Theorem~\ref{Lifespan of IIE}]
We divide the proof into several steps.
\begin{itemize}
    \item [(I)] We start with necessary estimates to establish the system of differential inequality, like \eqref{Boussinesq Differential Inequalities}. Consider the Duhamel term:
\begin{equation*}
\int_0^t g(\tau,X_\tau)d\tau=\nabla^\perp\rho_0 \int_0^t \nabla^* X_\tau (\nabla^* uu)(X_\tau)d\tau,
\end{equation*}
which admits the following estimates holding for generic $\delta\le 1$, thanks to Lemma \ref{Rho 0 Estimate}:
\begin{equation*}
 \bigg\lVert  \int_0^t g(\tau,X_\tau)d\tau \bigg\rVert_{1,p}\lesssim  \delta\int_0^\infty \lVert \nabla X_\tau\rVert_\infty \lVert\nabla u\rVert_\infty \lVert u\rVert_\infty d\tau+\delta  \int_0^t\lVert\nabla^2 X_\tau\rVert_p \lVert\nabla u\rVert_\infty \lVert u\rVert_\infty
\end{equation*}
\begin{equation*}
 +\delta\int_0^t\lVert\nabla X\rVert_\infty^2 \lVert\nabla^2 u\rVert_p\lVert \nabla u\rVert_\infty d\tau .  
\end{equation*}
Hence, we have:
\begin{equation*}
  \bigg\lVert  \int_0^t g(\tau,X_\tau)d\tau \circ A_t\bigg\rVert_{1,p} \lesssim  M\delta\Big[\int_0^t (\dot M+N\frac{\dot M}{M})\lVert u\rVert_\infty+\frac{d}{dt}M^2 \frac{\dot N}{N}d\tau
  \Big].
\end{equation*}
Therefore, Duhamel's formula yields the following estimate:
\begin{equation*}
 \lVert\omega\rVert_{1,p}\lesssim 1+M+  M\delta\Big[\int_0^t (\dot M+N\frac{\dot M}{M})\lVert u\rVert_\infty+M\dot M \frac{\dot N}{N}d\tau\Big],
\end{equation*}
and 
\begin{equation}\label{Energy norm of IIE}
\mathcal E(t)\lesssim  1+M +\delta M\Big[\int_0^t (\dot M+N\frac{\dot M}{M})\lVert u\rVert_\infty+M\dot M \frac{\dot N}{N}d\tau\Big]+\delta N.
\end{equation}
So, for any time $t\ge 0$ and any $\delta<1$, we have the $\mathbf{unconditional}$ estimate:
\begin{equation}\label{unconditional Duhamel estimate}
    \lVert\omega\rVert_{1,p}\lesssim 1+M+  \delta MQ,\quad Q=\int_0^t (\dot M+N\frac{\dot M}{M})\lVert u\rVert_\infty+\dot M \frac{\dot N}{N}d\tau .
\end{equation}
By Lemma \ref{Estimates of u}, we obtain that for $t\in [0,T_\delta]$ and $\delta<\delta_0\le (2C)^{-1}$:
\begin{equation*}
 \frac{\dot N}{N}\lesssim \frac{1+\delta N}{1-C\delta N} \bigg[1+M+  M\delta\Big(\int_0^t (\dot M+N\frac{\dot M}{M})\lVert u\rVert_\infty+M\dot M \frac{\dot N}{N}d\tau\Big) \bigg] . 
\end{equation*}
Moreover, taking $k=1$ in \eqref{u Wk,p Estimate}, we have
\begin{equation*}
\lVert u\rVert_\infty\lesssim \lVert u\rVert_{1,p}\lesssim \bigg(\frac{1+\delta M}{1-C\delta M}\bigg)\lVert\omega\rVert_p
\end{equation*}
\begin{equation*}
 \lesssim   \frac{1+\delta M}{1-C\delta M}\lVert\omega\rVert_\infty.
\end{equation*}
By noticing that the bounding constant in \eqref{unconditional Duhamel estimate} is some $C^\prime=C(\lVert\omega_0\rVert_{1,p})$, abusing the notation a little bit, we conclude that the following system of differential inequalities hold on $[0,T_\delta]$ for some $C$ depending on $\omega_0$ and the constant $C$ in Lemma \ref{Estimates of u}:
\begin{equation}\label{Differential Inequality System IIE}
\left\{
\begin{aligned}
&{\dot M}\le C  M\Big[\lVert\omega\rVert_\infty+\frac{\lVert\omega\rVert_\infty}{1-C\delta M}\Big]\Big(1+\log \Big(2+\frac{\dot N}{N}\Big)\Big);\\
&\dot N\le  CN\bigg(\frac{1+\delta N}{1-C\delta N}\bigg) \big(1+M+  \delta MQ \big);  \\
&\dot Q=  \Big(\dot M+N\frac{\dot M}{M}\Big)\lVert u\rVert_\infty+M\dot M \frac{\dot N}{N},\quad Q(0)=0 . 
\end{aligned}
\right.    
\end{equation}
Moreover, as long as $C\delta M(t)<1$, we have the following estimate for $\lVert\omega\rVert_\infty$:
\begin{equation*}
 \lVert\omega(t,\cdot)\rVert_\infty \le \lVert\omega_0\rVert_\infty+\lVert\nabla \rho_0\rVert_\infty\int_0^t \lVert\nabla^*X\rVert_\infty \lVert\nabla u\rVert_\infty \lVert u\rVert_\infty d\tau  
\end{equation*}
\begin{equation}\label{Linfty estimate for omega}
 \le   C+C\delta \int_0^t \frac{(1+\delta M)\dot M}{1-C\delta M}\lVert\omega\rVert_\infty d\tau.
\end{equation}

\item [(II)]

Now we propose our bootstrap scheme. We say $\mathscr H(t)$ is satisfied if
\begin{equation}\label{Hypothesis IIE}
 C\delta M\le 1/2,\quad C\delta N\le 1/2  ,\quad \delta Q\le 1.
\end{equation}
Assuming that $\mathscr H(T_*)$ holds, then clearly $T_*\le T_\delta^*$ and the estimates \eqref{W2p estimate for u} and \eqref{Linfty estimate for nabla u} hold. Since $M,N,Q$ are all increasing, we conclude that $\mathscr H(t)$ holds for all $0\le t\le T_*$. 

Moreover, for $t\in [0,T_*]$, \eqref{Linfty estimate for omega} implies that
\begin{equation*}
\lVert\omega(t,\cdot)\rVert_\infty\le   C+3C\delta \int_0^t {\dot M(\tau)\lVert\omega(\tau,\cdot)\rVert_\infty} d\tau.  
\end{equation*}
By the integral version of Gr\"onwall inequality, we have:
\begin{equation*}
 \lVert\omega(t,\cdot)\rVert_\infty\le  C+3C^2\delta\int_0^t \dot M(\tau)\exp[3C\delta(M(t)-M(\tau))]d\tau
\end{equation*}
\begin{equation*}
 \le C+2C\exp(3C\delta M(t))\le 3e^2 C .   
\end{equation*}

Using the fact that $C\ge 1$, we are able to reduce \eqref{Differential Inequality System IIE} $[0,T_*]$ to the following system:
\begin{equation}\label{Relaxed Differential Inequality System IIE}
\left\{
\begin{aligned}
&{\dot M}\le  9e^2CM\Big(1+\log\Big(1+\frac{\dot N}{N}\Big)\Big);\\
&\dot N\le 6CN\big(1+M);  \\
&\dot Q\le 9e^2C(\dot M+N\dot M/M)+M\dot M\dot N/N.
\end{aligned}
\right.    
\end{equation}

We now introduce the following outcome $\mathscr O(t)$: We say $\mathscr O(t)$ holds if 
 there exists $\Lambda_1,\Lambda_2,\Lambda_3,\Lambda_4>0$ such that 
\begin{equation*}
M(t)\le \exp(e^{\Lambda_1 t}),\quad N(t)\le \exp(\Lambda_2\exp(e^{\Lambda_1t})),\quad Q(t)\le \Lambda_4\exp(\Lambda_3\exp(e^{\Lambda_1 t}).
\end{equation*}
Apply Lemma \ref{Hierarchical Growth of Norms for Bou and IIE} to \eqref{Relaxed Differential Inequality System IIE}, with the same $M$,$N$,$Q$ and constant $9e^2C$, it follows that $\mathscr H(t)$ implies $\mathscr O(t)$. Since $M,N,Q$ are all absolutely continuous and increasing with $M(0)=N(0)=1$, $Q(0)=0$, there exists sufficiently small $t$ such that $\mathscr H(t)$ holds. On the other hand, the closedness condition trivially holds for $\mathscr O(t)$.

\item [(III)] We now close the bootstrap, which amounts to show (II) in lemma \ref{Bootstrap Lemma}. To this end, apply Lemma \ref{Bootstrap Closure Lemma} with
\begin{equation*}
  m=3,\quad\mathcal F_1=M,\quad \mathcal F_2=N,\quad \mathcal F_3=Q,\quad \kappa_1=\kappa_2=1/2,\quad \kappa_3=1.  
\end{equation*}
and
\begin{equation*}
 \zeta_1=\zeta_2=1,\quad \zeta_3=   \Lambda_4.
\end{equation*}
As a consequence, we have the bootstrap closed on $[0,T]$ with $T\ge T_\delta$. More precisely:
\begin{equation*}
C_0=\frac{99}{100\max\{1/2,\Lambda_4\}},\quad C_1=\max\{\Lambda_2,\Lambda_3\},\quad C_2=1,\quad C_3=\Lambda_1    
\end{equation*}
and $\delta_0$, $T_\delta$ respectively defined by \eqref{Delta0 Definition} and \eqref{Triple Log Lifespan}. As a consequence, we have for all $t\in [0,T_\delta]$:
\begin{equation*}
  \mathsf E(t)\lesssim 1+M+\delta MQ +\delta N\lesssim 1+M\lesssim 1+C_0\delta^{-1}. 
\end{equation*}
Which completes the proof of the theorem.

\end{itemize}

\end{proof}

\subsection{A Quantitative Continuation Criterion of Inhomogeneous Euler}

As suggested by \eqref{Relaxed Differential Inequality System IIE}, it is possible to control the blow up of IIE using $\lVert\omega\rVert_\infty$ only. Here, we present a quantitative continuation criteria of IIE based on the scalar quantity $\omega$, which is a 
 conditional version of \eqref{BKM L1Linfty Bound}.

\begin{thm}\label{Control of Energy Norm II}
Let $2<p<\infty$, let $(\omega_0,\rho_0)\in W^{1,p}\times W^{2,p}$. Let $\delta$ measuring inhomogeneity of the initial data be defined as:
\begin{equation*}
\delta:=\lVert\rho_0^{-1}-1\rVert_{W^{2,p}}  .  
\end{equation*}
Assume $(\omega,\rho)$ is a solution of \eqref{IIE vorticity} on $[0,T_*)$. Then, there exist constants $\delta_0$ and $\{C_i^*\}_{i=0}^3$ such that if $\delta<\delta_0$, and the following condition holds:
\begin{equation}\label{Quantitative Bound of vorticity integral}
 \int_0^{T} \lVert\omega(t,\cdot)\rVert_\infty dt \le C_3^*\log [C_2^* \log  (C_1^*\log ({C_4^*\delta^{-1}}) )],  
\end{equation}
then the solution can be continued past time $T$.
\end{thm}

\begin{proof}[Proof of Theorem~\ref{Control of Energy Norm II}]

We follow a similar bootstrapping argument as we did in section 3. We first propose the same bootstrap assumption $\mathscr H$ defined as in \eqref{Hypothesis IIE} and we have accordingly the differential inequality \eqref{Differential Inequality System IIE} holding for all $t\in [0,T_*]$, provided $\mathscr H(T_*)$ holds.

Hence, similar computation yields for some constant $C>0$:
    \begin{equation*}
        \frac{\dot M}{1+M}\le C\lVert\omega\rVert_\infty[1+\log(1+M)].
    \end{equation*}
    Changing the variable $\mathfrak M(t)=1+\log (1+M)$, straightforward computation yields:
    \begin{equation*}
     M(t)\le \exp\Big[\exp\Big(C\int_0^t\lVert\omega(\tau,\cdot)\rVert_\infty d\tau\Big)\Big]-1.
    \end{equation*}
    Now we set
    \begin{equation*}
     U(t)=\int_0^t \lVert\omega(\tau,\cdot)\rVert_\infty d\tau ,  
    \end{equation*}
    and we rewrite the $M$-bound to
    \begin{equation*}
      M(t)\le  \exp[\exp(CU(t))]-1.
    \end{equation*}
    Hence, we have
    \begin{equation*}
    N(t)\le  \exp[6C(1+M)]\le \exp[6C\exp e^{ CU(t)}].
    \end{equation*}

    Moreover, $\mathscr H(T_*)$ implies $\lVert\omega(t,\cdot)\rVert_\infty\le 9e^2C$ holds for all $t\in [0,T_*]$. Hence, by the proof of Lemma \ref{Hierarchical Growth of Norms for Bou and IIE}, we conclude the following upper bound for $Q$:
    \begin{equation*}
     Q(t)\le  30C\exp[12C\exp (e^{Ct})].
    \end{equation*}

    Now we define $\mathscr O(t)$ as the following
    \begin{equation*}
    M(t)\le  \exp[\exp(CU(t))]-1,\quad   N(t)\le \exp[6C\exp e^{ CU(t)}].
    \end{equation*}
   and
    \begin{equation*}
    Q(t)\le  30C\exp[12C\exp (e^{CU(t)})].
    \end{equation*}
    We now choose our numerical constants. Fix
    \begin{equation*}
     C_0=\frac{9}{300C},\quad  C_1^*=(12C)^{-1},\quad C_2^*=1,\quad C_3^*= C^{-1}.   
    \end{equation*}
    Our quantitative assumption \eqref{Quantitative Bound of vorticity integral} guarantees that
    \begin{equation}\label{U bound}
     U(t)\le \frac{1}{C}\log\log\frac{1}{12C}\log\frac{3}{100C\delta} .   
    \end{equation}
    Which implies the following:
    \begin{equation*}
     \delta Q(t)\le \frac{9}{10},\quad \delta CN(t)\le  C\delta\exp\Big(\frac{1}{2}\log\frac{3}{100C\delta}\Big)=\frac{\sqrt {3C\delta}}{10}.
    \end{equation*}
    and 
    \begin{equation*}
     C\delta M(t)\le \frac{\delta}{12} \log\frac{3}{100C\delta} . 
    \end{equation*}
It suffices to choose $\delta_0$ small enough such that 
\begin{equation*}
  g(\delta)=\frac{\delta}{12} \log\frac{3}{100C\delta}=\frac{\delta}{12}(\log 3-2\log 10-\log (C\delta))
\end{equation*}
is increasing on $(0,\delta_0)$, together with the following condition to ensure \eqref{U bound} is well-defined:
\begin{equation}\label{Validity Condition ofr U}
 \log\frac{1}{12C}\log\frac{3}{100 C\delta_0}>1.   
\end{equation}
Differentiating $g$ yields:
\begin{equation*}
   \log(3/100) -\log(C\delta)-C> 0\quad\Longrightarrow \quad \delta< \frac{3}{100 Ce^C}.
\end{equation*}
while \eqref{Validity Condition ofr U} implies:
\begin{equation}\label{Delta0 Estimate}
  \delta_0<  \frac{3}{100C\exp(12Ce)}.
\end{equation}
Hence, we choose $\delta_0$ that satisfies \eqref{Delta0 Estimate} to estimate $M$:
\begin{equation*}
  C\delta M(t)\le \frac{\delta}{12} \log\frac{3}{100C\delta} <\frac{3Ce}{1200C\exp(12Ce)}\le \frac{1}{1200C}<1/2.
\end{equation*}

Hence, we have checked all conditions in Lemma \ref{Bootstrap Lemma}. We are now ready to conclude: for any $\delta<\delta_0$ with $\delta_0$ satisfying \eqref{Delta0 Estimate}, the feedback loop closes on $[0,T]$ and the proof is completed, as we have the same estimate for $\mathsf E$ as in Theorem \ref{Lifespan of IIE}.
\end{proof}

\section{MHD Case}
We now move on to show the long-time existence of MHD equation in dimension $2$, when the magnetic field $B$ is close to $0$. Here, we fix $p>2$ and define the inhomogeneity measurement to be $\delta:=\lVert \rho_0-1\rVert_{3,p}$. We will have to require one more degree of regularity of $\rho_0$ compared to threshold regularity to guarantee the local well-posedness of $2$-d MHD, i.e. $\rho_0\in W^{4,p}$, which is crucial in our proof and matches the results obtained in \cite{CobbFanelli2023}. 

Consider the following vorticity formulation of MHD:
\begin{equation}\label{MHD Vorticity}
\left\{
\begin{aligned}
&\partial_t\omega+u\cdot\nabla\omega=\{\rho,\Delta\rho\},\\
&\partial_t\rho+u\cdot\nabla\rho=0,\\
&u=\nabla^\perp\Delta^{-1}\omega.
\end{aligned}
\right.
\end{equation}

It is known that if we switch to the vorticity-current variable $(\omega, J)$, $J=\Delta\rho$, we can rewrite the above system as the following symmetric form:
\begin{equation}\label{MHD Vorticity-Current}
\left\{
\begin{aligned}
&\partial_t\omega+u\cdot\nabla\omega=\{\rho, J\},\\
&\partial_t J+u\cdot\nabla J=\{\rho,\omega\}+\mathscr Q(\omega,J),\\
&u=\nabla^\perp\Delta^{-1}\omega,\quad J=\Delta\rho.
\end{aligned}
\right.
\end{equation}
Here, $\mathscr Q$ is a bilinear $0$-th order operator:
\begin{equation*}
\mathscr  Q(\omega,J)=2\mathsf{Tr}(\nabla\nabla^\perp\Delta^{-1}\omega\cdot \nabla\nabla^\perp \Delta^{-1}J)    .
\end{equation*}
We now introduce the Els\"asser variable, which diagonalizes the vorticity-current system. Let $\xi=\omega+J$, $\eta=\omega-J$. Then, the original equation becomes:
\begin{equation}\label{Elsasser System}
\left\{
\begin{aligned}
&\partial_t\xi+(u-\nabla^\perp\rho)\cdot\nabla\xi=\mathscr  Q(\omega,J),\\
&\partial_t \eta+(u+\nabla^\perp\rho)\cdot\nabla \eta=-\mathscr  Q(\omega,J).
\end{aligned}
\right.
\end{equation}

Applying the standard Calderon-Zygmund estimates to $\mathscr Q$ yields the following:
\begin{equation*}
\lVert\mathscr  Q(\omega,J)\rVert_{1,p}=2\lVert\mathsf{Tr}(\nabla\nabla^\perp\Delta^{-1}\omega\cdot \nabla\nabla^\perp \Delta^{-1}J)  \rVert_{1,p}\lesssim \lVert\omega\rVert_{1,p}\lVert  J\rVert_{1,p},
\end{equation*}
\begin{equation*}
\lVert\mathscr  Q(\omega,J)\rVert_{2,p}\lesssim \lVert\omega\rVert_{2,p}\lVert  J\rVert_{\infty}+\lVert\omega\rVert_\infty \lVert J\rVert_{2,p}.
\end{equation*}                                                                   
Change the variable into Lagrangian coordinates, we define
\begin{equation}\label{Lagrangian variables}
    \tilde\xi:=\xi\circ X_t,\quad \tilde\eta:=\eta\circ X_t,
\end{equation}
here $X$ is the flow map of $u$, and we let $A$ be the back to label map. 

The Els\"asser system has the advantage that it allows simple linear transport estimates to close in Sobolev spaces, thanks to the smoother driven velocity field and forcing term admitting the same level regularity as transported quantities. However, it buries the smallness of $\rho$, $B$ and $J$ in the close-to-Euler regime. Hence, we are forced to single out the magnetic field equation independently and handle the loss of derivative in vector transport. This is the key technical difficulty. The idea, similar to the treatment in \cite{CobbFanelli2023}, is to go to higher norm $\lVert u\rVert_{3,p}+\lVert\rho\rVert_{4,p}$, use commutator estimate to get a nuanced control of growth of $B$ based on the higher norm. Then, we can obtain a perturbed system of differential inequalities as Boussinesq and IIE to run a bootstrap argument. 

We set up several notations here. Define:
\begin{equation*}
 M:=\exp\bigg(\int_0^t \lVert\nabla u\rVert_\infty d\tau\bigg),\quad Y:=\lVert \xi\rVert_{1,p}+\lVert\eta\rVert_{1,p},\quad Z:=\lVert\xi\rVert_{2,p}+\lVert \eta\rVert_{2,p}.    
\end{equation*}

We also observe that for any $k\ge 1$, the following equivalence between norms hold:
\begin{equation}\label{equivalence of norms}
 \lVert\xi\rVert_{k,p}+\lVert \eta\rVert_{k,p}\approx \lVert\omega\rVert_{k,p}+\lVert J\rVert_{k,p}\approx \lVert u\rVert_{k+1,p}+\lVert B\rVert_{k+1,p}.   
\end{equation}
Hence, we have $\lVert u\rVert_{2,p}\lesssim Y$ and 
\begin{equation*}
 \lVert\xi_0\rVert_{k,p}, \lVert\eta_0\rVert_{k,p}\lesssim \lVert\omega_0\rVert_{k,p}+\lVert J_0\rVert_{k,p}\lesssim \lVert\omega_0\rVert_{k,p}+\lVert \rho_0\rVert_{4,p}\quad\text{for }k=1,2.
\end{equation*}
Therefore, we may harmlessly represent $A\le C(\lVert\xi_0\rVert_{2,p},\lVert\eta_0\rVert_{2,p})B$ as $A\lesssim B$, provided $C$ is increasing and $\lVert \rho_0\rVert_{4,p}$ is assumed to be bounded by some given numerical constant.

We now carry out the following lemma, which provides several \textit{a priori} estimates that we will need later.
\begin{lem}\label{Apriori Estimate for new Variables}
    Assume  $\rho_0\in W^{4,p}$ satisfies the conditions $\lVert \rho_0-1\rVert_{3,p}\le \delta<1$ and $\lVert\rho_0\rVert_{4,p}\le 100$, let $u_0\in W^{3,p}$ , then the following estimates hold:
    \begin{itemize}
        \item [(I)] The magnetic field $B=\nabla^\perp\rho$ satisfies
        \begin{equation}\label{Magnetic field Estimate}
         \frac{d}{dt} \lVert B\rVert_{2,p}\lesssim \lVert\nabla u\rVert_\infty \lVert B\rVert_{2,p}+ \lVert\nabla\rho\rVert_\infty\lVert u\rVert_{3,p}.     
        \end{equation}
        As a consequence, we have
        \begin{equation}\label{Magnetic field Estimate II}
         \lVert B(t,\cdot)\rVert_{2,p}\lesssim \delta M+\delta M\int_0^tZ(\tau)d\tau
        \end{equation}
        \item[(II)]     The Lagrangian variables $\tilde\xi$, $\tilde\eta$ defined in \eqref{Lagrangian variables} satisfies the following estimate:
    \begin{equation}\label{lagrangian variable estimates}
      \lVert\tilde\xi\rVert_{1,p},\lVert\tilde\eta\rVert_{1,p}\lesssim  1+\exp(Ct\delta)+\exp(2Ct\delta)\int_0^t\lVert\omega\rVert_{1,p}\lVert J\rVert_{1,p}d\tau. 
    \end{equation}
    As a consequence, we have:
    \begin{equation}\label{Elsasser variable estimates}
      Y\lesssim 1+M+\exp(Ct\delta)M+ \delta M^2\exp(2Ct\delta)\int_0^t Y(\tau)\bigg(1+\int_0^t Z(\tau)d\tau\bigg)d\tau.
    \end{equation}
    \item[(III)] We have the following higher norm estimate:
    \begin{equation}\label{higher lagrangian variable estimates}
     \lVert\tilde\xi\rVert_{2,p},\lVert\tilde\eta\rVert_{2,p}\lesssim  \exp\bigg(C\int_0^t Y(\tau)d\tau\bigg)+\int_0^t \exp\bigg(C\int_\tau^t Y(r)dr\bigg)Y(\tau)Z(\tau)d\tau  . 
    \end{equation}
    Which implies:
     \begin{equation}\label{higher Elsasser variable estimates}
      Z\lesssim \exp\bigg(C\int_0^t Y(r)d\tau\bigg).
    \end{equation}
    \end{itemize}

\end{lem}

\begin{proof}[Proof of Lemma~\ref{Apriori Estimate for new Variables}]
\begin{itemize}
We first remark that the proof works as long as $\lVert\rho_0\rVert_{4,p}$ is bounded, here we set a bound to remove the dependence of constants $C$ in $\lVert\rho_0\rVert_{4,p}$. However, although we have $u_0, B_0\in W^{3,p}$, the smallness assumption is really at the level of natural local well-posedness regularity class of MHD.

    \item [] 
    
    (I) We estimate $\lVert B\rVert_{2,p}=\lVert \nabla\rho\rVert_{2,p}$. Straightforward calculation implies:
    \begin{equation*}
      \lVert B\rVert_{2,p}^p=\sum_{j=1}^3\lVert \partial^j\rho\rVert_p^p=\sum_{j=1}^3\int_\Omega \lvert \partial^j \rho\rvert^p dx   . \end{equation*}
    Therefore:
    \begin{equation*}
     \frac{d}{dt} \lVert B\rVert_{2,p}^p=p\lVert B\rVert_{2,p}^{p-1}\frac{d}{dt} \lVert B\rVert_{2,p}=\sum_{j=1}^3\int_\Omega\frac{d}{dt} (\lvert D^j \rho\rvert^2)^{p/2} dx    
    \end{equation*}
    \begin{equation*}
     = \sum_{j=1}^3\int_\Omega p\lvert D^j \rho\rvert^{p-2}D^j\rho  D^j\partial_t \rho dx  
    \end{equation*}
    \begin{equation*}
     = -\sum_{j=1}^3\int_\Omega p\lvert D^j \rho\rvert^{p-2}D^j\rho  \partial^j(u\cdot\nabla \rho) dx  
    \end{equation*}
    \begin{equation*}
       = -\sum_{j=1}^3\int_\Omega p\lvert \partial^j \rho\rvert^{p-2}\partial^j\rho  [\partial^j,u\cdot\nabla] \rho dx -\underbrace{\sum_{j=1}^3\int_\Omega p\lvert \partial^j \rho\rvert^{p-2}\partial^j\rho  (u\cdot\nabla \partial^j \rho)dx}_{(*)}. 
    \end{equation*}
    Notice that:
    \begin{equation*}
    p\lvert D^j \rho\rvert^{p-2}D^j\rho  (u\cdot\nabla D^j \rho)=u\cdot\nabla(\lvert D^j\rho\rvert^p),    
    \end{equation*}
    and the last term $(*)$ vanishes due to divergence-free property of $u$. Moreover, by Kato-Ponce commutator estimate [cf. \cite{KatoPonce1988}\cite{Li2019}]:
    \begin{equation}\label{Kato-Ponce Commutator}
       \lVert[\partial^j,u\cdot\nabla] \rho  \rVert_{p}\lesssim \lVert \partial^j u\rVert_p\lVert \nabla\rho\rVert_\infty+\lVert\nabla u\rVert_\infty \lVert \nabla^3\rho\rVert_p.
    \end{equation}
Therefore, by H\"older inequality, we obtain
\begin{equation*}
 \lVert B\rVert_{2,p}^{p-1}\frac{d}{dt} \lVert B\rVert_{2,p}\le   \sum_{j=1}^3\int_\Omega \Big\lvert \lvert \partial^j \rho\rvert^{p-2}\partial^j\rho  [\partial^j,u\cdot\nabla] \rho \Big\rvert dx  
\end{equation*}
\begin{equation*}
 \lesssim    \sum_{j=1}^3\Big(\int_\Omega\lvert D^j\rho\rvert^{(p-1)\cdot q} dx\Big)^{1/q}\lVert[D^j,u\cdot\nabla] \rho  \rVert_{p}
\end{equation*}
\begin{equation*}
  \lesssim  \lVert B\rVert_{2,p}^{p/q}(\lVert u\rVert_{3,p}\lVert \nabla\rho\rVert_\infty+\lVert\nabla u\rVert_\infty \lVert \nabla^3\rho\rVert_p)
\end{equation*}
where $q$ is the H\"older conjugate of $p$. Hence, dividing both sides by $\lVert B\rVert_{2,p}^{p/q}$ and using $p/q=p-1$, we conclude: 
\begin{equation*}
 \frac{d}{dt} \lVert B\rVert_{2,p}\lesssim \lVert\nabla u\rVert_\infty \lVert B\rVert_{2,p}+ \lVert\nabla\rho\rVert_\infty\lVert u\rVert_{3,p}.  
\end{equation*}
   Therefore, we established \eqref{Magnetic field Estimate}.

   Now, rewrite \eqref{Magnetic field Estimate} in terms of $M$, we have
   \begin{equation*}
    \frac{d}{dt}  \lVert B\rVert_{2,p}\lesssim \frac{\dot M}{M}\lVert B\rVert_{2,p}+\delta MZ,
   \end{equation*}
   where we use  \eqref{Transport Estimate}. Now, we define the new variable $\mathcal B:=M^{-1}\lVert B\rVert_{2,p}$ and rewrite the above inequality as follows:
   \begin{equation*}
     \dot{\mathcal B}\lesssim \delta Z,\quad \mathcal B_0=M^{-1}(0)\lVert B_0\rVert_{2,p}\le \delta.
   \end{equation*}
   Integrating the inequality, we conclude:
   \begin{equation*}
    \mathcal B(t)\lesssim \lVert B_0\rVert_{2,p}+ \delta\int_0^tZ(\tau)d\tau\quad\Longrightarrow\quad  \lVert B(t,\cdot)\rVert_{2,p}\lesssim \delta M\bigg(1+\int_0^tZ(\tau)d\tau\bigg).
   \end{equation*}

    \item [](II) We notice by Duhamel's principle that $\tilde\xi$, $\tilde\eta$ are respectively transported by the flows $X_\mp$, defined as follows:
    \begin{equation*}
       \dot X_\mp=\mp B_0(X_\mp),\quad B_0=\nabla^\perp\rho_0 .
    \end{equation*}
    Denote by $A_\mp$ the corresponding back-to-label map, by chord-arc estimate we have:
    \begin{equation*}
     \lVert\nabla A_\mp\rVert_\infty\le \exp\bigg(\int_0^t\lVert\nabla B_0\rVert_\infty d\tau\bigg) \le  \exp\bigg(C\int_0^t\lVert\nabla\rho_0\rVert_{2,p} d\tau\bigg) \le \exp(Ct\delta).
    \end{equation*}
    
    Now, Duhamel's principle yields:
    \begin{equation*}
     \tilde\xi(t,x)=\bigg[\tilde\xi_0+\int_0^t  \mathscr Q(\omega,J)(X_-(\tau,\cdot))d\tau\bigg]\circ A_-(t,x) ;
    \end{equation*}
    \begin{equation*}
     \tilde\eta(t,x)=\bigg[\tilde\eta_0+\int_0^t  \mathscr Q(\omega,J)(X_+(\tau,\cdot))d\tau\bigg]\circ A_+(t,x) .
    \end{equation*}
    Standard computation yields:
    \begin{equation*}
     \lVert\tilde\xi\rVert_{1,p}\lesssim \lVert \tilde\xi_0\rVert_p+\lVert\nabla A_-\rVert_\infty\lVert \tilde\xi_0\rVert_p+\lVert\nabla A_-\rVert_\infty\int_0^t \lVert\nabla X_-(\tau)\rVert_\infty \lVert\nabla \mathscr Q\rVert_pd\tau
    \end{equation*}
    \begin{equation*}
   \lesssim     1+\exp(Ct\delta)+\exp(2Ct\delta)\int_0^t\lVert\omega\rVert_{1,p}\lVert  J\rVert_{1,p}d\tau.
    \end{equation*}
    Similar estimates yield the same bound:
      \begin{equation*}
     \lVert\tilde\eta\rVert_{1,p}\lesssim 1+\exp(Ct\delta)+\exp(2Ct\delta)\int_0^t\lVert\omega\rVert_{1,p}\lVert  J\rVert_{1,p}d\tau,
    \end{equation*}
therefore \eqref{lagrangian variable estimates} holds true.

   By Sobolev estimates of rearrangements \eqref{Transport Estimate}, we have
   \begin{equation*}
    \lVert\xi   \rVert_{1,p}\lesssim 1+\lVert\nabla X\rVert_\infty \lVert\tilde\xi\rVert_{1,p},\quad \lVert\eta   \rVert_{1,p}\lesssim 1+\lVert\nabla X\rVert_\infty \lVert\tilde\eta\rVert_{1,p},
   \end{equation*}
   which, together with the estimate $\lVert \omega\rVert_{1,p}\lesssim Y$ and \eqref{Magnetic field Estimate II}, demonstrates \eqref{Elsasser variable estimates}.
   \item [](III) We now track $\xi$, $\eta$ along Els\"asser characteristics: Define $\Phi_\pm(r,t)$ as the two-parameter flow map associated with div-free vector field $u{\pm}\nabla^\perp\rho$:
   \begin{equation*}
    \partial_t\Phi_{\pm}^{r,t}(a)=(u\pm \nabla^\perp \rho)(t,\Phi_\pm^{r,t}(a)),\quad \Phi_\pm^{r,r}(a)=a.   \end{equation*}
    The following semigroup property holds:
    \begin{equation*}
       \Phi_{\pm}^{r,s}\circ \Phi_{\pm}^{s,t}=\Phi_{\pm}^{r,t},\quad \forall 0\le r\le s\le t. 
    \end{equation*}
    We also need to make use of the spatial inverse of $\Phi_{\pm}^{r,t}$. Let us denote
    \begin{equation*}
      \Psi_{\pm}^{t,s}=(\Phi_{\pm}^{s,t})^{-1}.  
    \end{equation*}
    Hence we have 
    \begin{equation*}
      \Phi_{\pm}^{r,s} = \Phi_{\pm}^{r,t}\circ \Psi_{\pm}^{t,s}.
    \end{equation*}
Then, we notice:
\begin{equation}\label{Elssaser Variables along Characteristics}
 \xi_t=\xi_0\circ \Psi_-^{t,0}+\int_0^t \mathscr Q(\omega,J)\circ \Psi_-^{t,\tau} d\tau,\quad  \eta:=\eta_0\circ  \Psi_+^{t,0}-\int_0^t \mathscr Q(\omega,J)\circ \Psi_+^{t,\tau}d\tau.  
\end{equation}
Now, we use the Sobolev rearrangement estimate: For general function $f\in W^{2,p}$:
\begin{equation*}
   \lVert f\circ \Psi_-^{t,s}\rVert_{2,p}\le \lVert f\rVert_p+\lVert\nabla \Psi_-^{t,s}\rVert_\infty\lVert\nabla f\rVert_p+R_2,
\end{equation*}
where $R_2$ is the second-order differentiated term, reads:
\begin{equation*}
 R_2=\lVert\nabla(\nabla^*\Psi_-^{t,s}\nabla f\circ \Psi_-^{t,s})\rVert_p\lesssim \lVert \nabla^2\Psi_-^{t,s}\rVert_p \lVert\nabla f\rVert_\infty+\lVert\nabla\Psi_-^{t,s}\rVert_\infty^2\lVert\nabla^2f\rVert_p    
\end{equation*}
\begin{equation*}
 \lesssim  \exp\bigg(C\int_s^t \lVert\nabla(u-B)\rVert_{1,p}d\tau\bigg) \lVert f\rVert_{2,p} \lesssim \exp\bigg(C\int_s^t Y(\tau)d\tau\bigg)\lVert f\rVert_{2,p}.
\end{equation*}
Combining with the chord-arc bound \eqref{Exponential Bound for Deformation}:
\begin{equation*}
\lVert\nabla\Psi_-^{t,s}\rVert_\infty\le \exp\bigg(\int_s^t \lVert \nabla(u-B)\rVert_\infty d\tau\bigg)\le  \exp\bigg(\int_s^t \lVert \nabla(u-B)\rVert_{1,p}d\tau\bigg)\le \exp\bigg(\int_s^t Y(\tau)d\tau\bigg),   
\end{equation*}
where we use the Sobolev embedding $W^{1,p}\hookrightarrow L^\infty$, we conclude:
\begin{equation}
 \lVert f\circ \Psi_-^{t,s} \rVert_{2,p}\lesssim   \exp\bigg(C\int_s^t Y(\tau)d\tau\bigg)\lVert f\rVert_{2,p}.
\end{equation}
It is easy to see that the estimate of $\eta$ follows the same line as $\xi$. Hence, it suffices to estimate $\lVert\xi\rVert_{2,p}$. We notice that
   \begin{equation*}
  \lVert\mathscr Q(\omega,J)\rVert_{2,p}\lesssim \lVert\omega\rVert_{2,p}\lVert J\rVert_\infty+\lVert\omega\rVert_\infty\lVert J\rVert_{2,p}\lesssim YZ.   
   \end{equation*}  
Then, using \eqref{Elssaser Variables along Characteristics}, we have
\begin{equation*}
 \lVert \xi(t,\cdot)\rVert_{2,p}\lesssim \exp\bigg(C\int_0^t Y(\tau)d\tau\bigg)+\int_0^t \exp\bigg(C\int_\tau^t Y(r)dr\bigg)Y(\tau)Z(\tau)d\tau.
\end{equation*}
Similar estimate also holds for $\eta$, following the same line of argument. Summing up the inequalities, by definition of $Z$, we obtain:
\begin{equation}\label{Z estimate}
 Z(t)\lesssim   \exp\bigg(C\int_0^t Y(\tau)d\tau\bigg)+\int_0^t \exp\bigg(C\int_\tau^t Y(r)dr\bigg)Y(\tau)Z(\tau)d\tau.
\end{equation}
Now we define 
\begin{equation*}
\mathcal Z(t):=\exp\bigg(-C\int_0^t Y(\tau)d\tau\bigg)Z(t).
\end{equation*}
Then the inequality \eqref{Z estimate} can be reformulated in terms of $\mathcal Z$:
\begin{equation*}
 \mathcal Z(t)\lesssim 1+\int_0^t Y(\tau)\mathcal Z(\tau)d\tau.  
\end{equation*}
Thanks to the Gr\"onwall inequality, we conclude:
\begin{equation*}
  \mathcal Z(t)\lesssim \exp\bigg(\int_0^t Y(\tau)d\tau\bigg)\quad\Longrightarrow\quad Z(t)\lesssim \exp\bigg(C\int_0^t Y(\tau)d\tau\bigg).  
\end{equation*}
Therefore, the desired estimate \eqref{higher Elsasser variable estimates} follows.
    \end{itemize}
\end{proof}

Now, combining the extrapolation inequality \eqref{Logarithmic Sobolev} with the above estimates, we have:
\begin{equation*}
 \frac{\dot M}{M} =\lVert\nabla u\rVert_\infty\le (1+\log(2+\lVert\omega\rVert_{1,p}))\lVert\omega\rVert_\infty  
\end{equation*}
\begin{equation*}
 \lesssim    [1+\log(2+\lVert\xi\rVert_{1,p}+\lVert\eta\rVert_{1,p})](\lVert\xi\rVert_{\infty}+\lVert\eta\rVert_\infty)
\end{equation*}
Using the transport structure of $\xi$ and Sobolev embedding $W^{1,p}\hookrightarrow L^\infty$, we have:
\begin{equation*}
\lVert \xi\rVert_\infty=\lVert\tilde\xi\rVert_\infty\le \lVert\xi_0\rVert_\infty+\int_0^t \lVert\mathscr Q(\omega,J)\rVert_\infty d\tau   
\end{equation*}
\begin{equation}\label{Linfty Estimate for xi}
 \lesssim 1+ \int_0^t \lVert\omega\rVert_{1,p} \lVert J\rVert_{1,p}d\tau.
\end{equation}
And the same estimates hold for $\eta$ following the same argument. Therefore, we conclude the following system of inequality holds for $M$, $Y$ and $Z$ for short time, with sufficiently large $C$ depending on $\omega_0$:
\begin{equation}\label{Differential Inequality System MHD}
\left\{
\begin{aligned}
&\dot M\le  CM(1+\log (1+Y))(1+Q),\\
&Y\le  C+CM+C\exp(Ct\delta)M+ C M\exp(2Ct\delta)Q,\\
&Z\le C\exp\bigg(C\int_0^t Y(\tau)d\tau\bigg),\\
& \dot Q=\lVert u\rVert_{2,p}\lVert B\rVert_{2,p}\le C\delta MY\bigg(1+\int_0^t Z(\tau)d\tau\bigg),\quad Q(0)=0.
\end{aligned}
\right.    
\end{equation}

Now we state our main theorem in this section.
\begin{thm}\label{Lifespan of MHD}
 Assume $\lVert\rho_0\lVert_{4,p}\le 100$ and $\lVert\rho_0-1\rVert_{3,p}<\delta$, then there exist constants $\delta_0$ and $C_1^\prime,C_2^\prime,C_3^\prime,C_4^\prime>0$ depending on $\omega_0$, such that as long as $\delta<\delta_0$, the energy norm $Z$ stays bounded on $[0,T_\delta]$. Where
\begin{equation}\label{Time of Existence of MHD}
 T_\delta:=C_1^\prime\log [C_2^\prime\log (C_3^\prime\log( C_4^\prime \delta^{-1}))].   
\end{equation}
\end{thm}

\begin{proof}[Proof of theorem~\ref{Lifespan of MHD}]

We use a similar bootstrap argument to show the desired result. Our bootstrap assumption $\mathscr H(t)$ is said to be satisfied if
\begin{equation*}
Q(t)\le 1,\quad C\delta t\le 1.  
\end{equation*}
Since $Q(0)=0$ and $\dot Q>0$, for any $\delta$ fixed we can define
\begin{equation*}
 T_*=\sup\{ t:Q(t)\le 1, C\delta t\le 1\}.   
\end{equation*}
Then $\mathscr H(t)$ is satisfied for all $t\in [0,T_*]$. Now on $[0,T_*]$, the system \eqref{Differential Inequality System MHD} of differential inequalities becomes:
\begin{equation}\label{Differential Inequality System MHD, Simplified}
\left\{
\begin{aligned}
&\dot M\le 2 CM(1+\log (1+Y)),\\
&Y\le  C+C(e^2+e+1)M,\\
&Z\le C\exp\bigg(C\int_0^t Y(\tau)d\tau\bigg),\\
& \dot Q=\lVert u\rVert_{2,p}\lVert B\rVert_{2,p}\le C\delta MY\bigg(1+\int_0^t Z(\tau)d\tau\bigg),\quad Q(0)=0.
\end{aligned}
\right.    
\end{equation}
Define $\mathfrak M(t):=1+\log(1+M(t))$, following a similar argument as in the proof of Boussinesq and IIE case, we find that for $C_1:=6C^2$:  
\begin{equation*}
 \dot{\mathfrak M}(t)\le C_1(1+\mathfrak M(t))\quad \Longrightarrow\quad \mathfrak M(t)\le 2e^{C_1t}-1.
\end{equation*}
As a consequence, we obtain
\begin{equation*}
  M(t)\le \exp(2e^{C_1t})-1,\quad Y(t)\le 2Ce^2 \exp(2e^{C_1t}).
\end{equation*}
Now, double exponential growth of $Y$ implies the triple exponential growth of $Z$. More precisely:
\begin{equation*}
 Z(t)\le C\exp\bigg(2C^2e^2\int_0^t   \exp(2e^{C_1r})dr\bigg)\le C\exp(C_2\exp(2e^{C_1t})),\quad C_2:=\frac{C^2e^2}{C_1}=e^2/6.
\end{equation*}
Now, we define the bootstrap outcome $\mathscr O(t)$ to be:
\begin{equation*}
M(t)\le \exp(2e^{C_1t}),\quad Y(t)\le 2Ce^2 \exp(2e^{C_1t}),\quad Z(t) \le C\exp(C_2\exp(2e^{C_1t})) .
\end{equation*}
We have already demonstrate that $\mathscr H(t)$ implies $\mathscr O(t)$. It suffices to show the feedback loop closes for some chosen $\{C_i^\prime\}_{i=1}^4$ and corresponding $T_\delta$. 

To this end, we let
\begin{equation*}
C_1^\prime=C_1^{-1},\quad C_2^\prime=\frac{1}{4},\quad  C_3^\prime=C_2^{-1},\quad C_4^\prime=\frac{99}{100C_4}.   
\end{equation*}
Assume now $\mathscr O(T_*)$ holds for some $T_*\le T_{\delta_0}$, notice first that $Q$ satisfies the following triple exponential estimate:
\begin{equation*}
 Q(t)\le   2C^2e^2\delta \int_0^t\exp(4e^{C_1\tau})\bigg(1+\int_0^\tau C\exp(C_2\exp(2e^{C_1r}))dr\bigg)d\tau
\end{equation*}
\begin{equation*}
 \le 2C^2e^2\delta \int_0^t  \exp(4e^{C_1\tau})C_3\exp(C_2\exp(4e^{C_1\tau})))d\tau
\end{equation*}
\begin{equation*}
 \le   2C^2\delta e^2 C_3(4C_1C_2)^{-1} \exp(C_2\exp(4e^{C_1 t})):=C_4\delta\exp(C_2\exp(4e^{C_1 t})).
\end{equation*}
Here:
\begin{equation*}
 C_3= \frac{e}{2e^3C_1C_2},\quad C_4=  2C^2 e^2 C_3(4C_1C_2)^{-1} =\frac{C^2}{4C_1^2C_2^2}=\frac{1}{4C^2e^2}.
\end{equation*}

Now for all $0\le t\le T_*\le T_{\delta_0}$, increasing property of the exponential yields:
\begin{equation*}
 Q(t)\le    Q(T_{\delta_0})=1,\quad C\delta_0 t\le CC_1^{-1}\delta_0\log \frac{1}{4}\log C_2^{-1}\log\frac{1}{C_4^\prime\delta_0}:=f(\delta_0).
\end{equation*}
It suffices to find $\delta_0$ small enough such that $f^\prime(\delta)>0$ for all $\delta<\delta_0$ and $f(\delta_0)\le 1$, so the desired bound holds for all $\delta<\delta_0$ and the proof is thus completed. To this end, we let
\begin{equation*}
 \alpha(\delta):=\log(C_4^\prime\delta^{-1}),\quad \beta(\delta):= \log (C_2^{-1}\alpha(\delta)) ,\quad \gamma(\delta):= \log (\beta(\delta)/4).
\end{equation*}
Then we have:
\begin{equation*}
 \alpha^\prime(\delta)=- \delta^{-1},\quad \beta^\prime(\delta)=-\frac{1}{\delta\alpha(\delta)},\quad \gamma^\prime(\delta)=-\frac{1}{\delta\alpha(\delta)\beta(\delta)},\quad f(\delta)= CC_1^{-1}\delta\gamma(\delta).
\end{equation*}
Hence, we obtain
\begin{equation*}
 f^\prime(\delta)= CC_1^{-1}\bigg(\gamma(\delta)-\frac{1}{\alpha(\delta)\beta(\delta)} \bigg) .
\end{equation*}
So it suffices to let $\gamma\ge 2$ and $\alpha\beta\ge 1$. Hence, we may choose $\delta_0=C_4^\prime\exp(-C_2e^{4e^2})$, which gives
\begin{equation*}
\alpha(\delta_0)=C_2e^{4e^2}\ge 1,\quad \beta(\delta_0)= 4e^2\ge 1,\quad \gamma(\delta_0)=2.   \end{equation*}
Finally, we complete the proof by computing
\begin{equation*}
 f(\delta_0)=CC_1^{-1}\delta_0\gamma(\delta_0)=2CC_1^{-1}C_4^\prime\exp(-C_2e^{4e^2})=\frac{99}{1200 C^3e^2}\exp(-C_2e^{4e^2})< 1.  
\end{equation*}

\end{proof}

\bibliographystyle{siam}
\bibliography{Ref}

\end{document}